\documentclass[12pt,a4paper]{article}
\usepackage{graphicx} 
\usepackage{amsmath,amsthm,amsfonts,amssymb} 
\usepackage{txfonts,eucal,bm} 
\usepackage{tikz-cd} 
\newtheorem{thm}{Theorem}[section]
\newtheorem{prop}[thm]{Proposition}
\newtheorem{cor}[thm]{Corollary}
\newtheorem{lem}[thm]{Lemma}
\theoremstyle{remark}
\newtheorem{rem}[thm]{Remark}
\numberwithin{equation}{section}

\DeclareMathOperator\AGL{AGL}%
\DeclareMathOperator\AOrth{AO}%
\DeclareMathOperator\Char{Char}%
\DeclareMathOperator\diag{diag}%
\DeclareMathOperator\DOrth{\Delta O}%
\DeclareMathOperator\GL{GL}%
\DeclareMathOperator\id{id}%
\DeclareMathOperator\Orth{O}%
\DeclareMathOperator\spn{span}%
\newcommand{\x}{\times} 
\newcommand{\Trans}{\mathrm T}
\newcommand{\trenn}{{-\hspace{0pt}}}
\newcommand{\li}{\langle} 
\newcommand{\re}{\rangle} 
\newcommand{\lire}{\li\,\cdot\,,\cdot\,\re} 
\newcommand{\VQ}{(\vV,Q)} 
\newcommand{\Oweak}{\Orth'} 
\newcommand{\AOweak}{\AOrth'} 
\newcommand{\DOweak}{\DOrth'}
\newcommand{\FV}{{F\times\vV}} 

\newcommand{\eps}{\varepsilon}
\renewcommand{\phi}{\varphi}
\renewcommand{\theta}{\vartheta}
\renewcommand{\kappa}{\varkappa}

\newcommand{\cF}{{\mathcal F}}

\newcommand{\cL}{{\mathcal L}}

\newcommand{\cP}{{\mathcal P}}

\newcommand{\bA}{{\mathbb A}}

\newcommand{\bP}{{\mathbb P}}

\newcommand{\bR}{{\mathbb R}}

\newcommand{\vS}{{\bm S}}
\newcommand{\vT}{{\bm T}}

\newcommand{\vV}{{\bm V}}

\newcommand{\vX}{{\bm X}}

\newcommand{\va}{{\bm a}}

\newcommand{\vc}{{\bm c}}

\newcommand{\ve}{{\bm e}}
\newcommand{\vf}{{\bm f}}

\newcommand{\vo}{{\bm o}}
\newcommand{\vp}{{\bm p}}

\newcommand{\vr}{{\bm r}}
\newcommand{\vs}{{\bm s}}
\newcommand{\vt}{{\bm t}}
\newcommand{\vu}{{\bm u}}
\newcommand{\vv}{{\bm v}}
\newcommand{\vw}{{\bm w}}
\newcommand{\vx}{{\bm x}}
\newcommand{\vy}{{\bm y}}

\begin{document}
\sloppy
\author{Hans Havlicek\thanks{https://orcid.org/0000-0001-6847-1544}}
\title{Affine Metric Geometry and Weak Orthogonal Groups}
\date{}
\maketitle

\begin{abstract}\noindent%
By following the ideas underpinning the well{\trenn}established ``homogeneous
model'' of an $n${\trenn}dimensional Euclidean space, we investigate whether
the motion group or the weak motion group of an $n${\trenn}dimensional affine
metric space on a vector space $\vV$ over an arbitrary field admits a specific
faithful linear representation as weak orthogonal group of an
$(n+1)${\trenn}dimensional metric vector space. Apart from a few exceptions,
such a representation exists precisely when the metric structure on $\vV$ is
given by a quadratic form with a non{\trenn}degenerate polar form.
\end{abstract}

\section{Introduction}\label{se:intro}

There is a widespread literature on the problem of describing the motion group
of the Euclidean space $\bR^n$ (equipped with the standard inner product) by
means of a Clifford algebra. One of the known approaches makes use of the
so-called ``homogeneous model''. It is based upon the introduction of
homogeneous coordinates or, said differently, the embedding of $\bR^n$ in the
projective space $\bP(\bR^{n+1})$, and it fits into the following more general
construction: First, $\bR^n$ is equipped with an inner product of signature
$(p,n-p,0)$. Then the dual vector space of $\bR^{n+1}$, in symbols
$(\bR^{n+1})^*$, is equipped with an inner product of signature $(p,n-p,1)$,
and the corresponding Clifford algebra is being used. So, the inner product on
$(\bR^{n+1})^*$ is degenerate with a one{\trenn}dimensional radical. See, for
example, C.~G.~Gunn \cite{gunn-19a}, D.~Klawitter \cite{klaw-15a}, D.~Klawitter
and M.~Hagemann \cite{klaw+h-13a}. We also refer to J.~M.~Selig
\cite{seli-05a}, \cite{seli-22a}, where in the Euclidean case ($p=n$) the
signature $(0,n,1)$ is used instead of $(n,0,1)$. The cited sources contain a
wealth of references to previous work.
\par
We are interested in the generalisation of the above results to arbitrary
affine metric spaces of finite dimension. In Sections~\ref{se:prelim} and
\ref{se:transvections}, we collect some basic facts from linear algebra and we
establish auxiliary results about transvections and dilatations, which are
employed in Section~\ref{se:main}. Our starting point in Section~\ref{se:aff}
is the affine space $\bA(\vV)$ on a finite{\trenn}dimensional vector space
$\vV$ over an arbitrary field $F$. By analogy to the real case, we consider the
$F$-vector spaces $\vV^*$ (\emph{i.e.}\ the dual vector space of $\vV$), $\FV$
and $\FV^*$; we identify the latter with the dual vector space of $\FV$. Our
main tool is a faithful linear representation
$\beta\colon\AGL(\vV)\to\GL(\FV^*)$, where $\AGL(\vV)$ denotes the group of all
affinities of $\vV$ onto itself. Then we recall the notion of an affine metric
space $\bA\VQ$, which arises by equipping $\bA(\vV)$ with a quadratic form
$Q\colon \vV\to F$.
\par
In Section~\ref{se:main}, we address the main problem: Find all dyads of metric
vector spaces $\VQ$ and $(\FV^*,\tilde{Q})$ such that $\bA\VQ$ has a motion
group or a weak motion group whose $\beta$-image coincides with the weak
orthogonal group of $(\FV^*,\tilde{Q})$. The transformations of the latter
group allow for a neat description in terms of the corresponding Clifford
algebra. However, this topic is beyond the scope of the present note; see
\cite{havl-21b} for further details and an extensive bibliography.
Proposition~\ref{prop:AO-beta=Ow} provides solutions to the above problem under
the extra assumptions that, firstly, the polar form of $Q$ is
non{\trenn}degenerate and, secondly, $\tilde{Q}$ is a non-zero scalar multiple
of a quadratic form $Q^\uparrow$ arising from $Q$ by an explicit construction;
see Proposition~\ref{prop:Q-auf-ab}. The polar form of $Q^\uparrow$ has a
particular one{\trenn}dimensional radical and, moreover, $Q^\uparrow$ maps all
vectors of the radical to $0$. In Remarks~\ref{rem:literatur} and
\ref{rem:reflect}, we refer to closely related outcomes by F.~Bachmann
\cite{bach-73a}, E.~W.~Ellers \cite{elle-84a}, E.~W.~Ellers and H.~H\"{a}hl
\cite{elle+h-84a}, J.~Helmstetter \cite{helm-05a}, E.~M.~Schr\"{o}der
\cite{schroe-92a}, H.~Struve and R.~Struve \cite{stru+s-22a}, H.~Wolff
\cite{wolff-67a}, \cite{wolff-67b}. Remark~\ref{rem:koo} contains the
transition from $Q$ to $Q^\uparrow$ in terms of coordinates. Our main result is
Theorem~\ref{thm:umkehr}. Apart from quite a few exceptional cases, which occur
when both $\dim\vV$ and $|F|$ are ``very small'', there are only the solutions
as in Proposition~\ref{prop:AO-beta=Ow}. All exceptional cases are itemised in
Remark~\ref{rem:tab}, Tables~\ref{tab:1}--\ref{tab:4}; in some of these cases
the polar form of $Q$ fails to be non{\trenn}degenerate. In conclusion, we
switch to the ``projective point of view'' by going over to the projective
space $\bP(\FV^*)$. It will turn out that this merely leads us to yet another
(trivial) exceptional case, but otherwise does not give rise to new results.

\section{Preliminaries}\label{se:prelim}
Throughout this article, we consider only \emph{finite{\trenn}dimensional}
vector spaces over a (commutative) field $F$. In what follows, we fix our
notation and we collect some basic facts; see \cite[Ch.~II]{bour-98a},
\cite{gruen+w-77a}, \cite{havl-22b} and the sources listed below.
\par
Let $\vV$ be a vector space. We write $\vV^*$ for its \emph{dual vector space},
$\lire\colon {\vV^* \x \vV}\to F$ for the \emph{canonical pairing}, $\Char F$
for the characteristic of $F$ and we put $F^\times:=F\setminus\{0\}$. The zero
vector of $\vV$ (resp.\ $\vV^*$) is denoted by $\vo$ (resp.\ $\vo^*$). Each
subset $\vS\subseteq \vV$ determines its \emph{annihilator} $\vS^\circ
:=\bigl\{\va^*\in \vV^*\mid \li \va^*,\vs\re =0 \mbox{~for all~}
\vs\in\vS\bigr\}$, which is a subspace of $\vV^*$. In particular, given any
subspace $\vT$ of $\vV$, in symbols $\vT\leq\vV$, we have $\dim\vT^\circ =
\dim\vV-\dim\vT$. We consider $\vV$ as the dual vector space of $\vV^*$ by
identifying $\vx\in \vV$ with the linear form $\li\,\cdot\,,\vx\re\colon
\vV^*\to F$. In this way our results apply, \emph{mutatis mutandis}, to
$\vV^*$.
\par
Let $\tilde{\vV}$ be a vector space, too, and let $\eta\colon\tilde\vV\to\vV$
be a linear mapping. The \emph{transpose} of $\eta$ is given as
$\eta^\Trans\colon\vV^*\to\tilde\vV^*\colon \va^*\mapsto \va^*\circ\eta$. Thus,
for all $\va^*\in{\vV}^*$ and all $\tilde\vx\in\tilde\vV$, we have $\bigl\li
\eta^\Trans(\va^*),\tilde\vx\bigr\re = \bigl\li \va^*,\eta(\tilde\vx)\bigr\re$.
The mapping $\eta^\Trans$ is linear and satisfies $(\eta^\Trans)^\Trans =
\eta$. The image of $\eta^\Trans$ and the kernel of $\eta$ are related by
$\eta^\Trans({\vV}^*) = (\ker\eta)^\circ $.
\par
All linear bijections of $\vV$ onto itself form the \emph{general linear group}
$\GL(\vV)$. Any pair $(\vc^*,\vf)\in\vV^*\times\vV$ such that
$\li\vc^*,\vf\re\neq -1$ gives rise to the linear bijection
\begin{equation}\label{eq:delta}
    \delta_{\vc^*,\,\vf}\colon \vV\to\vV \colon \vx \mapsto \vx + \li\vc^*,\vx\re \vf .
\end{equation}
If $\vc^*=\vo^*$ or $\vf=\vo$, then $\delta_{\vc^*,\,\vf}$ equals the identity
$\id_\vV$. Otherwise, $\delta_{\vc^*,\,\vf}\neq\id_\vV$ fixes precisely the
vectors of the hyperplane $\ker\vc^*\leq\vV$ and $\delta_{\vc^*,\,\vf}$ is
called a \emph{transvection} (resp.\ \emph{dilatation}) provided that
$\li\vc^*,\vf\re = 0$ (resp.\ $\li\vc^*,\vf\re \neq 0$); see \cite{elle-90a},
\cite[p.~20]{tayl-92a}.\footnote{The term ``dilatation'' appears with a
different meaning, among others, in \cite[p.~26]{schroe-91b}.}
\par
Upon choosing a vector $\vf\in\vV\setminus\{\vo\}$, we put
\begin{equation}\label{eq:Delta}
    \Delta(\vV,\vf) := \bigl\{\delta_{\va^*,\,\vf}\mid \va^*\in\vV^*
    \mbox{~and~}
    \li\va^*,\vf\re\neq -1\bigr\} ,
\end{equation}
which is a subgroup of $\GL(\vV)$. It is easily checked that there is a
bijective mapping
\begin{equation}\label{eq:a-stern-bij}
    \bigl\{\va^*\in\vV^*\mid \li\va^*,\vf\re \neq -1\bigr\} \to
    \Delta(\vV,\vf)
    \colon \va^*\mapsto \delta_{\va^*,\,\vf} .
\end{equation}
\par
Next, let $Q\colon \vV\to F$ be a quadratic form. So $\VQ$ is a \emph{metric
vector space} as in \cite[1.1]{schroe-95a}; see also \cite[Ch.~IX]{bour-07a},
\cite[Sect.~1]{havl-21b}, \cite[\S~7]{schroe-92a}, \cite{tayl-92a}. Then
$B\colon\vV\times\vV\to F\colon (\vx,\vy)\mapsto Q(\vx+\vy)-Q(\vx)-Q(\vy)$
denotes the \emph{polar form} of $Q$. We have $B(\vx,\vx) = 2 Q(\vx)$ for all
$\vx\in\vV$. From $B$ being bilinear, we get
\begin{equation}\label{eq:D}
    D \colon \vV\to \vV^*\colon \vx \mapsto D(\vx):=B(\vx,\cdot\,)
\end{equation}
as the \emph{induced linear mapping} of $B$. The transpose of $D$ takes the
form $D^\Trans\colon \vV\to\vV^*$. Using \eqref{eq:D} and the fact that $B$ is
a symmetric bilinear form, it follows $\li D^\Trans(\vy),\vx\re = \li
D(\vx),\vy\re = B(\vx,\vy) = B(\vy,\vx) = \li D(\vy),\vx\re$ for all
$\vx,\vy\in \vV$. Hence we have
\begin{equation}\label{eq:D=DT}
    D = D^\Trans .
\end{equation}
Vectors $\vx,\vy\in\vV$ are \emph{orthogonal}, in symbols $\vx\perp\vy$,
precisely when $B(\vx,\vy)=0$. Given $\vS\subseteq\vV$ the set
$\vS^\perp:=\{\vx\in\vV\mid \vx\perp\vs \mbox{~for all~} \vs\in \vS\}$ is a
subspace of $\vV$. In particular, $\vV^\perp$ is called the \emph{radical} of
$B$. Then
\begin{equation}\label{eq:rad=ker}
    \vV^\perp = \bigl\{\vx\in\vV\mid \li D(\vx),\vy\re = 0
    \mbox{~for all~}\vy\in\vV\bigr\} =
    \ker D .
\end{equation}
Also, \eqref{eq:rad=ker} and \eqref{eq:D=DT} imply
\begin{equation}\label{eq:D(V)}
    (\vV^\perp)^\circ = (\ker D)^\circ = D^\Trans(\vV) = D(\vV) .
\end{equation}
The \emph{rank} of $B$ is defined as $\dim D(\vV)$. If $\Char F=2$, then $B$ is
an alternating bilinear form and its rank turns out to be even. The linear
mapping $D$ is bijective if, and only if, $\vV^\perp= \{\vo\}$. Under these
circumstances $B$ is said to be \emph{non{\trenn}degenerate}.
\par
The above notation ($Q$, $B$, $D$, $\perp$) will be maintained throughout. In
the presence of several quadratic forms, a common subscript or superscript will
be added to these symbols.
\par
Again, let a linear mapping $\eta\colon\tilde\vV \to \vV$ be given. Then the
\emph{pullback} of $Q$ along $\eta$, that is $Q\circ \eta$, is a quadratic
form, say $\tilde{Q}$, and $\tilde{B}(\tilde\vx,\tilde\vy) =
B\bigl(\eta(\tilde\vx),\eta(\tilde\vy)\bigr)$ for all
$\tilde\vx,\tilde\vy\in\tilde\vV$. The left hand side of the last equation can
be rewritten as $\bigl\li \tilde{D}(\tilde\vx),\tilde\vy\bigr\re$; the right
hand side equals $\bigl\li(D\circ\eta)(\tilde\vx), \eta(\tilde\vy)\bigr\re =
\bigl\li(\eta^\Trans \circ D\circ\eta)(\tilde\vx), \tilde\vy\bigr\re$. Hence
\begin{equation}\label{eq:D-tilde}
    \tilde{D} =\eta^\Trans\circ D\circ \eta .
\end{equation}
\par
A mapping $\phi\in\GL(\vV)$ is called an \emph{isometry} of $\VQ$ if
$Q=Q\circ\phi$. All isometries of $\VQ$ make up the \emph{orthogonal group}
$\Orth\VQ$. The formula
\begin{equation}\label{eq:D-phi}
    (\phi^\Trans)^{-1}\circ D = D\circ \phi
    \mbox{~~for all~~} \phi\in\Orth\VQ
\end{equation}
follows by replacing $\eta$ with $\phi$ in \eqref{eq:D-tilde} and by taking
into account $\tilde{D}=D$. The \emph{weak orthogonal group} $\Oweak\VQ$
consists of all isometries of $\VQ$ that fix the radical $\vV^\perp$
elementwise. The group $\Oweak\VQ$ appears in the literature under various
names; our terminology and notation follows \cite{elle-77a}. If $\vr\in\vV$
satisfies $Q(\vr)\neq 0$, then the \emph{$Q${\trenn}reflection} in the
direction of $\vr$, that is the mapping
\begin{equation}\label{eq:xi}
    \xi_\vr\colon\vV\to\vV\colon\vx\mapsto \vx - Q(\vr)^{-1} B(\vr,\vx)\vr ,
\end{equation}
belongs to $\Oweak\VQ$. If, moreover, $\vr\in \vV^\perp$, which implies $\Char
F=2$, then $D(\vr)=\vo^*$ and so $\xi_\vr=\id_\vV$. Otherwise, $\xi_\vr$ is of
order two. Each $\phi\in\Oweak\VQ$ is a product of $Q${\trenn}reflections,
unless $F$ and $\VQ$ satisfy one of the conditions \eqref{eq:Ow-ausn1} or
\eqref{eq:Ow-ausn2} for some basis $\{\ve_1,\ve_2,\ldots,\ve_n\}$ of $\vV$ and
all $\vx=\sum_{h=1}^{n}x_h\ve_h$ with $x_h\in F$:
\begin{align}
    |F|=2,\; \dim\vV > 2       &\mbox{~~and~~} Q(\vx) = x_1x_2 ;        \label{eq:Ow-ausn1}\\
    |F|=2,\; \dim\vV\geq 4     &\mbox{~~and~~} Q(\vx) = x_1x_2+x_3x_4 ; \label{eq:Ow-ausn2}
\end{align}
see \cite[Sect.~2]{havl-21b} for numerous references.\footnote{The conditions
on ``$\dim\vV$'' as in \eqref{eq:Ow-ausn1} and \eqref{eq:Ow-ausn2} have been
written down incorrectly in \cite{havl-21b}.}

\section{Lemmata on transvections and dilatations}\label{se:transvections}

Let $\VQ$ be a metric vector space. According to \eqref{eq:Delta}, any
$\vf\in\vV\setminus\{\vo\}$ gives rise to the group $\Delta(\vV,\vf)$, which in
turn determines a subgroup of the orthogonal group $\Orth\VQ$ and a subgroup of
the weak orthogonal group $\Oweak\VQ$, namely
\begin{multline}
\label{eq:Delta-O}
        \hspace{0.6cm}\DOrth(\vV,Q,\vf)  := \Delta(\vV,\vf)\cap\Orth\VQ \mbox{~~and~~}\\
        \DOweak(\vV,Q,\vf) := \Delta(\vV,\vf)\cap\Oweak\VQ .\hspace{0.6cm}
\end{multline}
Furthermore, if $Q(\vf)\neq 0$, then $\vf$ yields the $Q${\trenn}reflection
$\xi_\vf\in\Oweak\VQ$; see \eqref{eq:xi}. We proceed with an explicit
description of the groups appearing in \eqref{eq:Delta-O}.

\begin{lem}\label{lem:Delta-O}
Let $\vf$ be a non-zero vector of a metric vector space $\VQ$.
\begin{enumerate}\itemsep0pt\parsep0pt
\item\label{lem:Delta-O.a} If $\vf\notin\vV^\perp$ and $Q(\vf)\neq 0$, then
    $\DOrth(\vV,Q,\vf) = \DOweak(\vV,Q,\vf) =\{\id_\vV,\xi_\vf\}$ and
    $\xi_\vf\neq\id_\vV$.
\item\label{lem:Delta-O.b} If $\vf\notin\vV^\perp$ and $Q(\vf)=0$, then
    $\DOrth(\vV,Q,\vf) = \DOweak(\vV,Q,\vf) = \{\id_\vV\}$.
\item\label{lem:Delta-O.c} If $\vf\in\vV^\perp$ and $Q(\vf)\neq0$, then
    $\DOrth(\vV,Q,\vf)= \DOweak(\vV,Q,\vf) = \{\id_\vV\} $.
\item\label{lem:Delta-O.d} If $\vf\in\vV^\perp$ and $Q(\vf)=0$, then
    \begin{equation}
        \label{eq:Delta-O.d} \DOrth(\vV,Q,\vf) =\Delta(\vV,\vf)
        \mbox{~~and~~}\\
        \DOweak(\vV,Q,\vf) = \bigl\{\delta_{\va^*,\,\vf}\mid \va^*\in
        (\vV^\perp)^\circ \bigr\} .
    \end{equation}
Furthermore, by putting $n := \dim\vV$ and $k := \dim\vV^{\perp}$, it
follows
\begin{equation}\label{eq:Delta-O-ord.d}
    \bigl|\DOrth(\vV,Q,\vf)\bigr| = |F^\times| \cdot |F|^{n-1}
    \mbox{~~and~~}
    \bigl|\DOweak(\vV,Q,\vf)\bigr| = |F|^{n-k} .
\end{equation}
\end{enumerate}
\end{lem}
\begin{proof}
We pick any $\va^*\in\vV^*$ subject to $\li\va^*,\vf\re\neq -1$. Then
$\delta_{\va^*,\,\vf}\in\Orth\VQ$ is equivalent to
\begin{equation}\label{eq:astern-O}
    Q\bigl(\delta_{\va^*,\,\vf}(\vx)\bigr) - Q(\vx)
    = \li\va^*,\vx\re B(\vx,\vf) + \li\va^*,\vx\re^2 Q(\vf) = 0
    \mbox{~~for all~~} \vx\in\vV .
\end{equation}
Also, $\delta_{\va^*,\,\vf}\in\Oweak\VQ$ is satisfied if, and only if,
\eqref{eq:astern-O} holds alongside with
\begin{equation}\label{eq:astern-Ow}
    \va^*\in(\vV^\perp)^\circ .
\end{equation}
If $Q(\vf)\neq 0$, then \eqref{eq:delta}, \eqref{eq:D} and \eqref{eq:xi} show
\begin{equation}\label{eq:xi=delta}
    \xi_\vf = \delta_{\vc^*,\,\vf}, \mbox{~~where~~} \vc^* := -Q(\vf)^{-1} D(\vf) .
\end{equation}
The claims in \eqref{lem:Delta-O.a}--\eqref{lem:Delta-O.c} and
\eqref{eq:Delta-O.d} now follow easily from \eqref{eq:astern-O},
\eqref{eq:astern-Ow}, \eqref{eq:xi=delta}, $\vf\in\vV^\perp$ being equivalent
to $D(\vf)=\vo^*$ and $\id_\vV=\delta_{\vo^*,\,\vf}$. By virtue of the
bijection \eqref{eq:a-stern-bij}, the equations in \eqref{eq:Delta-O-ord.d} are
immediate from \eqref{eq:Delta}, $\dim (\vV^\perp)^\circ = n-k$ and
\eqref{eq:Delta-O.d}.
\end{proof}

Next, we present a crucial result about a particular subgroup of the group
$\Delta(\vV,\vf)$; compare Corollary~\ref{cor:trans}.

\begin{lem}\label{lem:trans}
Let $\vf$ be a non-zero vector of a metric vector space $\VQ$. Then the
following are equivalent.
\begin{enumerate}\itemsep0pt\parsep0pt
\item\label{lem:trans.a} The group $\bigl\{\delta_{\va^*,\,\vf}\mid
    \va^*\in \{\vf\}^\circ \bigr\}$ is contained in $\Oweak\VQ$.
\item\label{lem:trans.b} One of the subsequent conditions holds:
    \begin{gather}
        Q(\vf)=0
        \mbox{~~and~~} \vV^\perp = F\vf; \label{eq:trans}\\
        \dim\vV = 1 ; \label{eq:trans=1}\\
        \dim\vV = 2 ,\; Q(\vf)\neq 0 ,\;  \dim\vV^\perp= 0 %
        \mbox{~~and~~} |F| = 2. \label{eq:trans=2}
    \end{gather}
\end{enumerate}
\end{lem}
\begin{proof}
\eqref{lem:trans.a}~$\Rightarrow$~\eqref{lem:trans.b} If $\dim\vV\leq 1$, then
\eqref{eq:trans=1} holds due to $\vf\neq\vo$. Otherwise, $\{\vf\}^\circ $ is at
least one{\trenn}dimensional and so $|F| \leq \bigl|\{\vf\}^\circ\bigr|$. The
bijection \eqref{eq:a-stern-bij} shows that the group
$\bigl\{\delta_{\va^*,\,\vf}\mid \va^*\in \{\vf\}^\circ \bigr\}$ is of order
$\bigl|\{\vf\}^\circ\bigr|$. By our assumption,
$\bigl\{\delta_{\va^*,\,\vf}\mid \va^*\in \{\vf\}^\circ \bigr\} \subseteq
\DOweak(\vV,Q,\vf)$; therefore
\begin{equation}\label{eq:ordnung}
    2\leq |F|
    \leq \bigl|\{\vf\}^\circ\bigr|
      =  \bigl|\bigl\{\delta_{\va^*,\,\vf}\mid \va^*\in \{\vf\}^\circ \bigl\}\bigr|
    \leq \bigl|{\DOweak(\vV,Q,\vf)}\bigr| .
\end{equation}
\par
\emph{Case} 1: $Q(\vf) = 0$. Then \eqref{eq:ordnung} implies that $\vf$ meets
the hypotheses of Lemma~\ref{lem:Delta-O}~\eqref{lem:Delta-O.d}. Hence
$\vf\in\vV^\perp$ and so $F\vf\leq\vV^\perp$. On the other hand, the second
equation in \eqref{eq:Delta-O.d} yields $\{\vf\}^\circ \leq (\vV^\perp)^\circ$.
This gives, by going over to annihilators on either side, $F\vf \geq
\vV^\perp$. All in all, \eqref{eq:trans} holds.
\par
\emph{Case} 2: $Q(\vf)\neq 0$. Now \eqref{eq:ordnung} implies that $\vf$ meets
the hypotheses of Lemma~\ref{lem:Delta-O}~\eqref{lem:Delta-O.a}. Hence
$\vf\notin\vV^\perp$. Furthermore, $\bigl|{\DOweak(\vV,Q,\vf)}\bigr| = 2$ and
so, together with \eqref{eq:ordnung}, we get $|F| = \bigl|\{\vf\}^\circ\bigr| =
2$. Consequently, $\dim {\{\vf\}^\circ} =1$, whence
$(F\vf)^\circ=\{\vf\}^\circ$ results in $\dim\vV = \dim (F\vf) + \dim
(F\vf)^\circ = 2$. From $\Char F =2$ and $\dim\vV=2$ being even, the radical
$\vV^\perp$ has even dimension $\leq 2$. Due to $\vf\notin\vV^\perp$, we cannot
have $\dim \vV^\perp=2$. Thus $\dim\vV^\perp=0$. To sum up, \eqref{eq:trans=2}
is satisfied.
\par
\eqref{lem:trans.b}~$\Rightarrow$~\eqref{lem:trans.a} First, suppose that
\eqref{eq:trans} holds. Then Lemma~\ref{lem:Delta-O}~\eqref{lem:Delta-O.d}
applies together with $(\vV^\perp)^\circ =\{\vf\}^\circ$. By the second
equation in \eqref{eq:Delta-O.d}, we have
$\bigl\{\delta_{\va^*,\,\vf}\mid\va^*\in\{\vf\}^\circ\bigr\} \subseteq
\Oweak\VQ$.
\par
Next, suppose that \eqref{eq:trans=1} holds. Here the group
$\bigl\{\delta_{\va^*,\,\vf} \mid \va^*\in \{\vf\}^\circ\bigr\}$ coincides with
$\{\id_\vV\}$ and so it is contained in $\Oweak\VQ$.
\par
Finally, suppose that \eqref{eq:trans=2} holds. Due to $|F|=2$, $Q(\vf)\neq 0$
actually means $Q(\vf)=-1=1$. From $\vV^\perp=\{\vo\}$ and $\Char F=2$, the
polar form $B$ of $Q$ is non{\trenn}degenerate and alternating. Since
$\dim\vV=2$ and $|F|=2$, the annihilator $\{\vf\}^\circ$ comprises only $D(\vf)
= -Q(\vf)^{-1} B(\vf,\cdot\,)$ and $\vo^*$. Thus, by $\delta_{\vo^*,\,\vf} =
\id_\vV$ and \eqref{eq:xi}, $\bigl\{\delta_{\va^*,\,\vf}\mid
\va^*\in\{\vf\}^\circ\bigr\} = \{\id_\vV,\xi_\vf\}\subseteq\Oweak\VQ$.
\end{proof}

The following lemma will take us to Corollary~\ref{cor:trans-skal}, which will
be used in the proof of Theorem~\ref{thm:proj}.

\begin{lem}\label{lem:trans-skal}
Let $\vf$ be a non-zero vector of a metric vector space $\VQ$. Then, for all
$s\in F\setminus\{0,1\}$ and all non-zero $\va^*\in\{\vf\}^\circ$, the mapping
$s\delta_{\va^*,\,\vf}\in\GL(\vV)$ does not belong to $\Oweak\VQ$.
\end{lem}
\begin{proof}
Suppose, by way of contradiction, that $s\delta_{\va^*,\,\vf}\in\Oweak\VQ$ with
$s$ and $\va^*$ as above. So, $\dim\vV\geq 2$ and $|F|>2$. The only eigenvalue
of the transvection $\delta_{\va^*,\,\vf}$ is $1\in F$ and the corresponding
eigenspace equals the hyperplane $\ker\va^*$ of $\vV$. Consequently, $s\neq
0,1$ is the only eigenvalue of $s\delta_{\va^*,\,\vf}$ and so the radical
$\vV^\perp$, which is fixed elementwise under $s\delta_{\va^*,\,\vf}$, turns
out to be $\{\vo\}$. From $s\delta_{\va^*,\,\vf}(\vf)=s\vf$ and
$s\delta_{\va^*,\,\vf}\in\Oweak\VQ$, the hyperplane $\{\vf\}^\perp$ coincides
with $s\delta_{\va^*,\,\vf}\bigl(\{\vf\}^\perp\bigr)$. A hyperplane of $\vV$ is
fixed (as a set) under $s\delta_{\va^*,\,\vf}$ if, and only if, it contains
$\vf$. Therefore $\vf\in\{\vf\}^\perp$ or, said differently, $B(\vf,\vf)=0$.
\par
\emph{Case} 1: $Q(\vf) \neq 0$. Then $0=B(\vf,\vf)=2 Q(\vf)$ forces $\Char F =
2$. Now $Q \bigl(s\delta_{\va^*,\,\vf}(\vf)\bigr) = s^2 Q(\vf) = Q(\vf)$ shows
$s^2=1$. This implies $s=1$, an absurdity.
\par
\emph{Case} 2: $Q(\vf) = 0$. As $\vf\notin\vV^\perp=\{\vo\}$ and
$\va^*\neq\vo^*$, there exists a vector $\vu\in\vV$ with $B(\vf,\vu)=1$ and
$\li \va^*,\vu\re\neq 0$. We put $\vv:=\vu-Q(\vu)\vf$, so that
$\delta_{\va^*,\,\vf}(\vv) = \vu + \bigl(\li\va^*,\vu\re - Q(\vu)\bigr)\vf$.
Then, by straightforward calculations, $Q(\vv)=0$ and
$Q\bigl(s\delta_{\va^*,\,\vf}(\vv)\bigr) = s^2 \li\va^*,\vu\re\neq 0$. This
contradicts $s\delta_{\va^*,\,\vf}$ being an isometry.
\end{proof}

Our final lemma relies on a result by E.~M.~Schr\"{o}der
\cite[(1.25)~Satz]{schroe-86a}. It will be an essential tool for proving
Theorem~\ref{thm:umkehr}.

\begin{lem}\label{lem:Orth1=2}
Let $(\vV,Q_1)$ be a metric vector space such that $\vV^{\perp_1}=\{\vo\}$.
Furthermore, suppose that none of the subsequent conditions applies:
\begin{align}
    \dim\vV = 1 \mbox{~~and~~}|F|=3 \label{eq:Orth1=2-ausn1}; \\
    \dim\vV = 2 \mbox{~~and~~}|F|=2 \label{eq:Orth1=2-ausn2}.
\end{align}
If a quadratic form $Q_2\colon\vV\to F$ satisfies
\begin{equation}\label{eq:Orth1=2}
    \Orth(\vV,Q_1)=\Orth(\vV,Q_2)
    \mbox{~~or~~}
    \Orth(\vV,Q_1)=\Oweak(\vV,Q_2) ,
\end{equation}
then $Q_1=cQ_2$ for some $c\in F^\times$.
\end{lem}
\begin{proof}
Our first goal is to establish that, whenever $\dim\vV\geq 1$, any
$Q_1${\trenn}reflection is also a $Q_2${\trenn}reflection and vice versa. To
this end, let us pick any vector $\vf\in\vV\setminus\{\vo\}$, whence $\vf\notin
\vV^{\perp_1}$. Also, for $i\in\{1,2\}$, we put (within this proof only)
\begin{equation*}
    d_i(\vf):=\bigl|\DOrth(\vV,Q_i,\vf)\bigr|
    \mbox{~~and~~}
    d_i'(\vf):=\bigl|\DOweak(\vV,Q_i,\vf)\bigr| ;
\end{equation*}
compare \eqref{eq:Delta-O}.
\par
\emph{Case} 1: $Q_1(\vf)\neq 0$. Then $\vf$ and $Q_1$ satisfy the hypotheses of
Lemma~\ref{lem:Delta-O}~\eqref{lem:Delta-O.a}, which gives $d_1(\vf)=2$. Hence
\eqref{eq:Orth1=2} implies
\begin{equation}\label{eq:d2(f)=2}
    d_2(\vf)=2 \mbox{~~or~~} d_2'(\vf)=2 .
\end{equation}
We use this intermediate result in order to find out which of the hypotheses
appearing in Lemma~\ref{lem:Delta-O}
\eqref{lem:Delta-O.a}--\eqref{lem:Delta-O.d} are met by $\vf$ and $Q_2$.
\par
Obviously, the hypotheses of \eqref{lem:Delta-O.b} do not hold, since under
these circumstances we would get $d_2(\vf) = d_2'(\vf) = 1$, a contradiction to
\eqref{eq:d2(f)=2}. Likewise, the hypotheses of \eqref{lem:Delta-O.c} cannot be
fulfilled. We claim that the hypotheses of \eqref{lem:Delta-O.d} are not
satisfied either. For a verification, we assume that the contrary holds. So
\eqref{eq:Delta-O-ord.d}, with $n := \dim\vV$ and $k := \dim\vV^{\perp_2}$,
gives $d_2(\vf) = |F^\times| \cdot |F|^{n-1}$ and $d_2'(\vf) = |F|^{n-k}$. If
$d_2(\vf)=2$, then either \eqref{eq:Orth1=2-ausn1} or \eqref{eq:Orth1=2-ausn2}
holds; both cases are contradictory, since they have been excluded. Thus
\eqref{eq:d2(f)=2} means $d_2'(\vf)=2$, whence $|F|=2$ and $n-k=1$. Then, due
to $\Char F=2$, the polar form of $Q_2$ is alternating and so its rank $n-k=1$
turns out to be even, which is also contradictory.
\par
By the above, $\vf$ and $Q_2$ are compelled to satisfy the hypotheses of
Lemma~\ref{lem:Delta-O}~\eqref{lem:Delta-O.a}, that is $\vf\notin\vV^{\perp_2}$
and $Q_2(\vf)\neq 0$. Consequently, $d_2(\vf)=d_2'(\vf)=2$ and so, by
\eqref{eq:Orth1=2} and Lemma~\ref{lem:Delta-O}~\eqref{lem:Delta-O.a}, the
$Q_1${\trenn}reflection in the direction of $\vf$ coincides with the
$Q_2${\trenn}reflection in the direction of $\vf$.
\par
\emph{Case} 2: $Q_1(\vf)=0$. So there is no $Q_1${\trenn}reflection in the
direction of $\vf$. Clearly, $\vf$ and $Q_1$ satisfy the hypotheses of
Lemma~\ref{lem:Delta-O}~\eqref{lem:Delta-O.b}, whence $d_1(\vf)=1$. Now
\eqref{eq:Orth1=2} implies
\begin{equation}\label{eq:d2(f)=1}
    d_2(\vf)=1 \mbox{~~or~~} d_2'(\vf)=1 .
\end{equation}
Since $\vf\notin\vV^{\perp_1}=\{\vo\}$, there is an auxiliary vector
$\vu\in\vV$ with $B_1(\vf,\vu)\neq 0$. Also, $Q_1(\vf)=0$ implies
$B_1(\vf,\vf)=0$. Therefore $B_1(\vf,\vf+\vu)=B_1(\vf,\vu)\neq 0$ and
$Q_1(\vf+\vu)=B_1(\vf,\vu)+Q_1(\vu)\neq Q_1(\vu)$. We put $\vv:=\vu$ if
$Q_1(\vu)\neq 0$ and $\vv:=\vf+\vu$ otherwise. Thus $Q_1(\vv)\neq 0$ and
$B_1(\vf,\vv)\neq 0$. Consequently, the $Q_1${\trenn}reflection in the
direction of $\vv$ does not fix $\vf$. From the previous case, this
$Q_1${\trenn}reflection is also a $Q_2${\trenn}reflection, which in turn
entails $\vf\notin\vV^{\perp_2}$. Therefore and by \eqref{eq:d2(f)=1}, $\vf$
and $Q_2$ satisfy the hypotheses of
Lemma~\ref{lem:Delta-O}~\eqref{lem:Delta-O.b}. Hence $Q_2(\vf)=0$ and a
$Q_2${\trenn}reflection in the direction of $\vf$ does not exist either.
\par
Finally, let us verify our assertion concerning $Q_1$ and $Q_2$: If
$\dim\vV=0$, then $Q_1=Q_2$ is the zero form and so $Q_1=cQ_2$ holds for
$c:=1$. Otherwise, by the above, the set of all $Q_1${\trenn}reflections
coincides with the set of all $Q_2${\trenn}reflections and, clearly,
$\{\vo\}=\vV^{\perp_1}\neq \vV$. Under these premises,
\cite[(1.25)~Satz]{schroe-86a} (see also \cite[(7.81)~Satz]{schroe-92a},
\cite[1.7.4]{schroe-95a}) establishes $Q_1=cQ_2$ for some $c\in F^\times$.
\end{proof}

\section{Affine metric spaces}\label{se:aff}

Throughout this section $\vV$ denotes a vector space. First, we collect some
well-known results about affine spaces and affine mappings. Our terminology is
close to the one in \cite[p.~33]{buek+c-95a} and \cite[Ch.~2,
Ch.~3]{gruen+w-77a}. For proofs we refer also to \cite[Ch.~II]{bour-98a},
\cite[Ch.~6]{havl-22b}, \cite[\S~5]{schroe-91b}, \cite[Ch.~2]{tayl-92a}, even
though the terminology from there may be different from ours.
\par
If $\vu\in\vV$ and $\vT\leq\vV$, then the coset $\vu+\vT$ will be addressed as
an \emph{affine subspace}\footnote{Some authors consider also the empty set as
being an affine subspace of $\vV$. We refrain from following this convention.}
of $\vV$. The \emph{affine space} $\bA(\vV)$ is understood to be the set
comprising all affine subspaces of $\vV$. The \emph{dimension} $\dim \bA(\vV)$
is defined as $\dim \vV$. The cosets of subspaces $\vT\leq\vV$ with dimension
$0$, $1$, $2$ and $\dim\vV-1$ are the \emph{affine points}, \emph{affine
lines}, \emph{affine planes} and \emph{affine hyperplanes} of $\vV$. Given any
$\vx\in\vV$ we shall usually write $\vx$ for the affine point $\vx+\{\vo\}$.
Also, we shall drop the adjective ``affine'' when speaking about points if no
confusion is to be expected. By analogy to the above, each affine subspace
$\vu+\vT$ of $\vV$ gives rise to the \emph{affine space} $\bA(\vu+\vT)$. It
comprises all affine subspaces of $\vV$ that are contained in $\vu+\vT$, and we
put $\dim\bA(\vu+\vT):=\dim \vT$.
\par
Let $\tilde{\vV}$ also be a vector space. We consider affine spaces
$\bA(\vu+\vT)$ and $\bA(\tilde\vu+\tilde\vT)$ with $\vu\in\vV$, $\vT\leq\vV$,
$\tilde\vu\in\tilde\vV$ and $\tilde\vT\leq\tilde\vV$. A mapping $\gamma\colon
\vu+\vT\to\tilde\vu+\tilde\vT$ is said to be \emph{affine} provided that it can
be written in the form
\begin{equation}\label{eq:gamma-allg}
   \gamma\colon \vu+\vT\to\tilde\vu+\tilde\vT\colon \vx \mapsto
   \gamma(\vw) + \gamma_{+}(\vx-\vw)
\end{equation}
for some point $\vw\in\vu+\vT$ and some linear mapping $\gamma_{+}\colon \vT
\to \tilde\vT$. An \emph{affinity} is understood to be a bijective affine
mapping.
\par
Let us briefly recall a few properties of the affine mapping $\gamma$ appearing
in \eqref{eq:gamma-allg}: Under $\gamma$, the affine space $\bA(\vu+\vT)$ is
mapped into the affine space $\bA(\tilde\vu+\tilde\vT)$. We have
$\gamma_{+}(\vx-\vy)=\gamma(\vx)-\gamma(\vy)$ for all $\vx,\vy\in\vu+\vT$, so
that $\gamma_{+}$ is uniquely determined by $\gamma$, whereas any point of
$\bA(\vu+\vT)$ may take over the role of $\vw$ in \eqref{eq:gamma-allg}. Also,
the affine mapping $\gamma$ is bijective if, and only if, $\gamma_{+}$ is a
linear bijection.
\par
The group of all affinities of $\vV$ onto itself is denoted by $\AGL(\vV)$ and
acts faithfully on $\bA(\vV)$. Any $\gamma\in\AGL(\vV)$ can be written in a
\emph{unique} way as
\begin{equation}\label{eq:gamma}
    \gamma \colon \vV\to\vV\colon \vx\mapsto \vt + \gamma_{+}(\vx)
    \mbox{~~with~~}\gamma_{+}\in\GL(\vV)
    \mbox{~~and~~} \vt\in\vV.
\end{equation}
Indeed, it suffices to rewrite \eqref{eq:gamma-allg} with $\vw:=\vo\in\vV$ and
$\vt:=\gamma(\vo)$. In particular, \eqref{eq:gamma} defines a
\emph{translation} if, and only if, $\gamma_{+} = \id_\vV$.
\par
In order to obtain a linear representation of the group $\AGL(\vV)$, we change
over from $\vV$ to the affine hyperplane $\{1\}\times\vV =
(1,\vo)+\{0\}\times\vV$ of the vector space $F\times\vV$. Thereby we make use
of the affinity
\begin{equation}\label{eq:epsilon}
    \eps\colon \vV\to \{1\}\times\vV \colon \vx \mapsto (1,\vo)+(0,\vx) = (1,\vx) .
\end{equation}
If $\gamma$ is given as in \eqref{eq:gamma}, then
$\eps\circ\gamma\circ\eps^{-1}$ is an affinity of $\{1\}\times\vV$.
Furthermore,
\begin{equation}\label{eq:gamma-zeta}
    \gamma^{\,\zeta}\colon \FV\to \FV\colon (x_0,\vx) \mapsto \bigl(x_0, x_0 \vt+\gamma_{+}(\vx)\bigr)
\end{equation}
is the \emph{only} linear mapping of $\FV$ to itself that extends
$\eps\circ\gamma\circ\eps^{-1}$. This $\gamma^{\,\zeta}$ is bijective. We
therefore obtain that
\begin{equation}\label{eq:zeta}
    \zeta\colon\AGL(\vV)\to\GL(\FV)
    \colon
    \gamma \mapsto \gamma^{\,\zeta}
\end{equation}
is a faithful linear representation of $\AGL(\vV)$. Its image will be written
as $\AGL(\vV)^{\,\zeta}$. The pairing
\begin{equation*}
    \lire\colon (\FV^*)\times (\FV) \to F \colon
    \bigl((a_0,\va^*),(x_0,\vx)\bigr)\mapsto a_0x_0 + \li \va^*,\vx \re
\end{equation*}
allows us to consider $\FV^*$ as being the dual vector space of $\FV$. There is
another faithful linear representation of $\AGL(\vV)$, which is known as the
\emph{dual} of \eqref{eq:zeta}; see \cite[p.~4]{fult+h-13a}. It reads
\begin{equation}\label{eq:beta}
    \beta\colon\AGL(\vV) \to \GL(\FV^*)
    \colon
    \gamma\mapsto \gamma^{\,\beta} := \bigl((\gamma^{\,\zeta})^\Trans\bigr)^{-1}
\end{equation}
and we denote its image by $\AGL(\vV)^{\,\beta}$. If $\gamma$ is given as in
\eqref{eq:gamma}, then
\begin{equation}
    \label{eq:gamma-beta}
    \gamma^{\,\beta}(a_0,\va^*)
    = \Bigl(a_0 - \bigl\li (\gamma_{+}^\Trans)^{-1}(\va^*),\vt\bigr\re,
    (\gamma_{+}^\Trans)^{-1}(\va^*)\Bigr) \\
    \mbox{~~for all~~}
    (a_0,\va^*)\in\FV^* .
\end{equation}
\begin{lem}\label{lem:AGL-rep}
$\AGL(\vV)^{\,\beta}$ is the elementwise stabiliser of $F(1,\vo^*)$ in
$\GL(\FV^*)$.
\end{lem}
\begin{proof}
From \eqref{eq:gamma-beta}, any $\gamma^{\,\beta}\in\AGL(\vV)^{\,\beta}$ fixes
all linear forms in $F(1,\vo^*)$. Conversely, any mapping belonging to
$\GL(\FV^*)$ can be written as $(\kappa^\Trans)^{-1}$ with $\kappa\in\GL(\FV)$.
If $(\kappa^\Trans)^{-1}$ fixes $F(1,\vo^*)$ elementwise, then
\begin{equation*}
    \bigl\li(1,\vo^*),\kappa(1,\vo) \bigr\re
    = \bigl\li\kappa^\Trans (1,\vo^*),(1,\vo) \bigr\re
    = \bigl\li(1,\vo^*) ,(1,\vo) \bigr\re
    = 1
\end{equation*}
implies $\kappa(1,\vo)=(1,\vt)$ for some $\vt\in\vV$. From $\{0\}\times\vV
=\ker (1,\vo^*)$ being invariant under $\kappa$, there exists a
$\gamma_{+}\in\GL(\vV)$ such that $\kappa(0,\vx)=\bigl(0,\gamma_{+}(\vx)\bigr)$
for all $\vx\in\vV$. The affinity $\gamma$ arising from $\gamma_{+}$ and $\vt$
according to \eqref{eq:gamma} satisfies
$\gamma^{\,\beta}=(\kappa^\Trans)^{-1}$.
\end{proof}

\begin{rem}\label{rem:aff-f}
The vector space $\FV^*$ can be identified with the vector space consisting of
all affine functions $\vV\to F$ as follows: Any $(a_0,\va^*)\in \FV^*$ is taken
for the affine function $\vV\to F \colon \vx\mapsto a_0 + \li\va^*,\vx\re$. The
linear forms belonging to $F(1,\vo^*)$ turn into the constant functions $\vV\to
F$. Furthermore, the faithful linear representation
$\beta\colon\AGL(\vV)\to\GL(\FV^*)$ in \eqref{eq:beta} may be described in the
following way: For any $\gamma\in\AGL(\vV)$, the image of the affine function
$(a_0,\va^*)$ under $\gamma^{\,\beta}$ is given by the product function
$(a_0,\va^*)\circ\gamma^{-1}$; see \eqref{eq:gamma-beta}.
\end{rem}

The \emph{projective space} $\bP(\FV)$ is understood to be the set of all
subspaces of $\FV$. The (projective) \emph{dimension} of $\bP(\FV)$ is one less
than the dimension of $\FV$. We adopt the usual geometric terms: \emph{points},
\emph{lines}, \emph{planes} and \emph{hyperplanes} of $\bP(\FV)$ are the
subspaces of $\FV$ with (vector) dimension $1$, $2$, $3$ and $\dim(\FV)-1$,
respectively; see \cite[p.~30]{buek+c-95a}, \cite[Ch.~2, Ch.~3]{gruen+w-77a}.
Furthermore, we refer to \cite[Ch.~II \S~9]{bour-98a}, \cite[Ch.~6]{havl-22b},
\cite[\S~6]{schroe-91b}, \cite[Ch.~3]{tayl-92a}. The general linear group
$\GL(\FV)$ acts in a canonical way on $\bP(\FV)$: Any $\kappa\in\GL(\FV)$
determines a \emph{projective collineation} on $\bP(\FV)$, which is given as
$\vX\mapsto\kappa(\vX)$ for all $\vX\in\bP(\FV)$. This action of $\GL(\FV)$ has
the kernel $F^\times \id_{\FV}$.
\par
Using the affinity $\eps$ as in \eqref{eq:epsilon}, the \emph{embedding} of the
affine space $\bA(\vV)$ in the projective space $\bP(\FV)$ takes the form
\begin{equation*}
    \iota\colon \bA(\vV)\to\bP(\FV)\colon \vx+\vT \mapsto \spn\bigl(\eps(\vx+\vT)\bigr)
    =\spn\bigl(\{1\}\times(\vx+\vT)\bigr).
\end{equation*}
An element of $\bP(\FV)$ is said to be \emph{at infinity} if it is not in the
image of $\iota$. In particular, $\{0\}\times\vV$ is the only \emph{hyperplane
at infinity} of $\bP(\FV)$. The group $\AGL(\vV)^{\,\zeta}$ (see
\eqref{eq:zeta}) acts on $\bP(\FV)$ as a group of projective collineations,
which allows us to deal with affinities in projective terms. Furthermore, by
sending any subspace of $\FV$ to its annihilator, a bijection of $\bP(\FV)$
onto $\bP(\FV^*)$ is obtained, which reverses inclusions. For example, the
hyperplane at infinity corresponds to the point $F(1,\vo^*)$. The action of the
group $\AGL(\vV)^{\,\beta}$ (see \eqref{eq:beta}) on the point set of
$\bP(\FV^*)$ mirrors the action of $\AGL(\vV)$ on the set of affine hyperplanes
of $\vV$.
\par
A quadratic form $Q\colon\vV\to F$ makes $\bA(\vV)$ into an \emph{affine metric
space} $\bA\VQ$, which can be equipped with a wealth of additional structure
\cite[\S~9]{schroe-92a}, \cite[Sect.~3]{schroe-95a}. If an affinity
$\mu\in\AGL(\vV)$ is given by analogy to \eqref{eq:gamma}, but with
$\mu_{+}\in\Orth\VQ$ (resp.\ $\mu_{+}\in\Oweak\VQ$) and arbitrary $\vt\in\vV$,
then $\mu$ is called a \emph{motion} (resp.\ a \emph{weak motion}) of $\VQ$.
All such motions (resp.\ weak motions) comprise the \emph{motion group}
$\AOrth\VQ$ (resp.\ the \emph{weak motion group} $\AOweak\VQ$).\footnote{There
is no widely accepted terminology for these groups.}
\par
Given a point $\vp$ in $\bA(\vV)$ and a vector $\vr\in\vV$ such that
$Q(\vr)\neq 0$, the mapping
\begin{equation}\label{eq:aff-xi}
    \xi_{\vp,\vr}\colon\vV\to\vV\colon \vx\mapsto \vx - Q(\vr)^{-1} B(\vr,\vx-\vp)\vr
\end{equation}
is the \emph{affine $Q${\trenn}reflection} with axis $\vp+\{\vr\}^\perp$ in the
direction of $\vr$; see \cite[p.~99]{schroe-92a} or \cite[p.~976]{schroe-95a},
where the term \emph{affine $Q${\trenn}symmetry} is used instead. Then, with
the notation as in \eqref{eq:gamma}$, (\xi_{\vp,\vr})_{+}=\xi_\vr$ and $\vt =
Q(\vr)^{-1}B(\vr,\vp)\vr$. This establishes $\xi_{\vp,\vr}\in\AOweak\VQ$.
\par
From \eqref{eq:zeta} and \eqref{eq:beta}, we obtain---by restriction---faithful
linear representations of $\AOrth\VQ$ and $\AOweak\VQ$. Their images are
written as $\AOrth\VQ^{\,\zeta}$, $\AOrth\VQ^{\,\beta}$, $\AOweak\VQ^{\,\zeta}$
and $\AOweak\VQ^{\,\beta}$.

\section{Main results}\label{se:main}

Let $\vV$ be a vector space. Thus there is the faithful linear representation
$\beta\colon\AGL(\vV)\to\GL(\FV^*)$ as in \eqref{eq:beta}. Also, we recall our
standard notation $(Q,B,D,\perp)$ for dealing with quadratic forms; see
Section~\ref{se:prelim}. Our aim is to find \emph{all} dyads of metric vector
spaces $\VQ$ and $(\FV^*,\tilde{Q})$ such that the $\beta$-image of the motion
group $\AOrth\VQ$ or the $\beta$-image of the weak motion group $\AOweak\VQ$
coincides with the weak orthogonal group $\Oweak(\FV^*,\tilde{Q})$. That is, we
require that one of the following holds:
\begin{align}
      \AOrth\VQ^{\,\beta} =& \Oweak(\FV^*,\tilde{Q}) ; \label{eq:oder1} \\
     \AOweak\VQ^{\,\beta} =& \Oweak(\FV^*,\tilde{Q}) . \label{eq:oder2}
\end{align}
The next two propositions provide \emph{most, but not all} of the solutions of
the above problem. Therein, we make use of two auxiliary linear mappings
together with their transposes:
\begin{equation*}
\begin{aligned}
    \nu&\colon\vV\to\FV   \colon \vx\mapsto(0,\vx),
    &\nu^\Trans&\colon \FV^*\to\vV^*\colon (a_0,\va^*)\mapsto \va^* ,\\
    \pi&\colon \FV \to\vV \colon (x_0,\vx) \mapsto \vx ,
    &\pi^\Trans&\colon
    \vV^*\to \FV^*\colon  \va^*\mapsto (0,\va^*) .
\end{aligned}
\end{equation*}

\begin{prop}\label{prop:Q-auf-ab}
Let $\vV$ be a vector space.
\begin{enumerate}\itemsep0pt\parsep0pt
\item\label{prop:Q-auf-ab.a} If $Q\colon\vV\to F$ is a quadratic form with
    non{\trenn}degenerate polar form $B$, then the quadratic form
    \begin{equation}\label{eq:Q-auf}
        Q^\uparrow:=Q\circ D^{-1}\circ\nu^\Trans
        \colon \FV^* \to F\colon (a_0,\va^*)
        \mapsto (Q\circ D^{-1})(\va^*)
    \end{equation}
satisfies $Q^\uparrow(1,\vo^*)=0$, and the polar form ${B^\uparrow}$ of
$Q^\uparrow$ has $F(1,\vo^*)$ as its radical.
\item\label{prop:Q-auf-ab.b} If $\tilde{Q}\colon\FV^*\to F$ is a quadratic
    form with $\tilde{Q}(1,\vo^*)=0$ and with $F(1,\vo^*)$ being the
    radical of its polar form $\tilde{B}$, then the linear mapping
    $\pi\circ \tilde{D}\circ\pi^\Trans\colon \vV^*\to\vV$ is invertible,
    and the quadratic form
\begin{equation}\label{eq:Q-ab}
    \tilde{Q}^\downarrow
    := \tilde{Q}\circ\pi^\Trans\circ(\pi\circ\tilde{D}\circ\pi^\Trans)^{-1}
    \colon \vV\to F
    \colon \vx \mapsto \tilde{Q}\bigl(0, (\pi\circ \tilde{D}\circ\pi^\Trans)^{-1}(\vx)\bigr)
\end{equation}
has a non{\trenn}degenerate polar form $\tilde{B}^\downarrow$.
\item\label{prop:Q-auf-ab.c} If $Q$ is given as in \eqref{prop:Q-auf-ab.a},
    then $(Q^\uparrow)^\downarrow=Q$ and $cQ^\uparrow=(cQ)^\uparrow$ for
    all $c\in F^\times$.
\item\label{prop:Q-auf-ab.d} If $\tilde{Q}$ is given as in
    \eqref{prop:Q-auf-ab.b}, then
    $(\tilde{Q}^\downarrow)^\uparrow=\tilde{Q}$ and $c\tilde{Q}^\downarrow=
    (c\tilde{Q})^\downarrow$ for all $c\in F^\times$.
\end{enumerate}
\end{prop}
\begin{proof}
\eqref{prop:Q-auf-ab.a} From \eqref{eq:rad=ker}, $\ker D=\vV^\perp=\{\vo\}$.
So, $D$ is bijective and $Q^\uparrow$ is well defined:
\begin{equation*}
    \begin{tikzcd}
       \vV   \arrow[r,"{Q}"] &
       F &
       \FV^* \arrow[l,"{Q^\uparrow}"']
             \arrow[ld,"{\nu^\Trans}"] \\
         &
       \vV^* \arrow[lu,"{D^{-1}}"]
    \end{tikzcd}
\end{equation*}
It is clear that $Q^\uparrow(1,\vo^*)=Q(\vo)=0$. Substituting $\eta :=
D^{-1}\circ\nu^\Trans$ and $\tilde{D} := {D^\uparrow}$ in \eqref{eq:D-tilde}
and then using $(D^{-1})^{\Trans}= (D^{\Trans})^{-1}= D^{-1}$, which follows
from \eqref{eq:D=DT}, gives
\begin{equation}\label{eq:Dauf}
    {D^\uparrow}=\nu\circ D^{-1}\circ\nu^\Trans \colon
    \FV^*\to \FV
    \colon
    (a_0,\va^*)\mapsto \bigl(0,D^{-1}(\va^*)\bigr)
\end{equation}
as the induced linear mapping of ${B^\uparrow}$. From \eqref{eq:Dauf}, $\ker
{D^\uparrow} = F(1,\vo^*)$. According to \eqref{eq:rad=ker}, the latter kernel
equals the radical of ${B^\uparrow}$.
\par
\eqref{prop:Q-auf-ab.b} By our assumption on the radical of $\tilde{B}$ and
from \eqref{eq:rad=ker}, $\ker \tilde{D} = F(1,\vo^*)$. Thus $\tilde{D}(\FV^*)
= \tilde{D}\Bigl(F(1,\vo^*)\oplus \bigl(\{0\}\times \vV^*\bigr)\Bigr) =
\tilde{D}\bigl(\{0\}\times \vV^*\bigr)$. On the other hand, from
\eqref{eq:D(V)}, $\tilde{D}(\FV^*) = \bigl( F(1,\vo^*)\bigr)^\circ =
\{0\}\times \vV$, whence $\tilde{D}\bigl(\{0\}\times \vV^*\bigr)= \{0\}\times
\vV$. Therefore $(\pi\circ\tilde{D}\circ\pi^\Trans)(\vV^*) = \vV$. Hence the
inverse mapping $(\pi\circ \tilde{D}\circ\pi^\Trans)^{-1}$ exists and
$\tilde{Q}^\downarrow$ is well defined:
\begin{equation*}
    \begin{tikzcd}
       \vV   \arrow[r,"{\tilde{Q}^\downarrow}"]
             \arrow[rd,"{(\pi\,\circ\,\tilde{D}\,\circ\,\pi^\Trans)^{-1}}"'] &
       F &
       \FV^* \arrow[l,"{\tilde{Q}}"']   \\
         &
       \vV^* \arrow[ru,"{\pi^\Trans}"']
    \end{tikzcd}
\end{equation*}
As $\tilde{D}$ coincides with its transpose, so does
$(\pi\circ\tilde{D}\circ\pi^\Trans)^{-1}$. By analogy with \eqref{eq:D-tilde},
\begin{equation}\label{eq:Dab}
    \tilde{D}^\downarrow
    =
    \bigl((\pi\circ \tilde{D}\circ\pi^\Trans)^{-1}
    \circ \pi\bigr)
    \circ\tilde{D}\circ
    \bigl(\pi^\Trans
    \circ (\pi\circ \tilde{D}\circ\pi^\Trans)^{-1}\bigr)
    =  (\pi\circ \tilde{D}\circ\pi^\Trans)^{-1} .
\end{equation}
The radical of $\tilde{B}^\downarrow$ equals $\ker\tilde{D}^\downarrow=\{\vo\}$
and so $\tilde{B}^\downarrow$ is non{\trenn}degenerate.
\par
\eqref{prop:Q-auf-ab.c} The first assertion follows from \eqref{eq:Q-auf},
\eqref{eq:Q-ab}, \eqref{eq:Dauf}, $\nu^\Trans\circ\pi^\Trans = \id_{\vV^*}$ and
$\pi\circ\nu=\id_\vV$:
\begin{equation*}
    (Q^\uparrow)^\downarrow
    = {\underbrace{Q\circ D^{-1}\circ\nu^\Trans}_{=\,Q^\uparrow}}
    \circ
    \pi^\Trans
    \circ
    \bigl(\pi\circ
    {\underbrace{\nu\circ D^{-1}\circ\nu^\Trans}_{=\,D^\uparrow}}
    \circ
    \pi^\Trans\bigr)^{-1}
    = Q .
\end{equation*}
The second assertion holds trivially.
\par
\eqref{prop:Q-auf-ab.d} By our assumptions on $\tilde{Q}$ and $\tilde{B}$, we
have $\tilde{Q} = \tilde{Q}\circ\pi^\Trans\circ\nu^\Trans$. So, from
\eqref{eq:Q-ab}, \eqref{eq:Q-auf} and \eqref{eq:Dab}, we obtain
\begin{equation*}
    (\tilde{Q}^\downarrow)^\uparrow
    = {\underbrace{\tilde{Q}\circ\pi^\Trans\circ(\pi\circ\tilde{D}\circ\pi^\Trans)^{-1}}_{=\,\tilde{Q}^\downarrow}}
    \circ
    {\underbrace{(\pi\circ\tilde{D}\circ\pi^\Trans)}_{=\,(\tilde{D}^\downarrow)^{-1}}}
    \circ\nu^\Trans
    = \tilde{Q} \circ\pi^\Trans\circ\nu^\Trans
    =\tilde{Q} .
\end{equation*}
Again, the second assertion holds trivially.
\end{proof}

\begin{prop}\label{prop:AO-beta=Ow}
Let $\VQ$ be a metric vector space such that the polar form $B$ of $Q$ is
non{\trenn}degenerate, let $Q^\uparrow$ be given as in \eqref{eq:Q-auf} and let
$c\in F^\times$. Then, with $\beta$ is as in \eqref{eq:beta},
$\AOrth\VQ^{\,\beta}=\Oweak\bigl(\FV^*,cQ^\uparrow\bigr)$.
\end{prop}
\begin{proof}
From $\Oweak\bigl(\FV^*,Q^\uparrow\bigr) =
\Oweak\bigl(\FV^*,cQ^\uparrow\bigr)$, it suffices to verify the claim for
$c=1$. Let any motion $\mu\in\AOrth\VQ$ be given by analogy to
\eqref{eq:gamma}, that is with $\mu_{+}\in\Orth\VQ$ and $\vt\in\vV$. Then
$\mu^{\,\beta}\in\Orth\bigl(\FV^*,Q^\uparrow\bigr)$ follows from
\begin{equation*}
    {\underbrace{Q\circ D^{-1}\circ\nu^\Trans}_{=\,Q^\uparrow}}\circ \mu^{\,\beta}
        = Q\circ D^{-1}\circ (\mu_{+}^\Trans)^{-1}\circ\nu^\Trans\\
        = Q\circ \mu_{+}\circ D^{-1}\circ\nu^\Trans
         = {\underbrace{Q\circ D^{-1}\circ\nu^\Trans}_{=\,Q^\uparrow}}
         ;
\end{equation*}
thereby we argue as follows. First, we use \eqref{eq:Q-auf}. Next, we read off
from \eqref{eq:gamma-beta}, with $\gamma:=\mu$, that
$\nu^\Trans\circ\mu^{\,\beta} = (\mu_{+}^\Trans)^{-1}\circ \nu^\Trans$. Then we
take into account $D^{-1}\circ (\mu_{+}^\Trans)^{-1} = \mu_{+} \circ D^{-1}$,
which follows from rewriting \eqref{eq:D-phi} with $\phi:=\mu_{+}$. By our
assumption, $Q\circ \mu_{+}=Q$ and, finally, we apply \eqref{eq:Q-auf} again.
From Lemma~\ref{lem:AGL-rep}, the radical of ${B^\uparrow}$ is elementwise
invariant under $\mu^{\,\beta}$. Hence
$\mu^{\,\beta}\in\Oweak\bigl(\FV^*,Q^\uparrow\bigr)$.
\par
Conversely, let any isometry belonging to $\Oweak(\FV^*,Q^\uparrow)$ be given.
This isometry fixes $(1,\vo^*)$ and, by Lemma~\ref{lem:AGL-rep}, it can be
written in the form $\gamma^{\,\beta}$ with $\gamma\in\AGL(\vV)$ as in
\eqref{eq:gamma}, that is with $\gamma_{+}\in\GL(\vV)$ and $\vt\in\vV$. It
remains to establish that $Q\circ\gamma_{+}=Q$. The calculation
\begin{equation*}
\begin{aligned}
    Q \circ \gamma_{+} & =  Q \circ \gamma_{+}\circ\pi\circ\nu
         = Q \circ \pi\circ\gamma^{\,\zeta}\circ{\underbrace{\nu \circ D^{-1}\circ\nu^\Trans}_{=\,D^\uparrow}}\circ\pi^\Trans\circ D\\
        &= Q \circ \pi\circ\nu \circ D^{-1}\circ\nu^\Trans\circ\gamma^{\,\beta}\circ\pi^\Trans\circ D
         \\
        &= {\underbrace{Q\circ D^{-1}\circ\nu^\Trans}_{=\,Q^\uparrow}}\circ\gamma^{\,\beta}\circ\pi^\Trans\circ D
         = Q                                                  \circ D^{-1}\circ\nu^\Trans\circ\pi^\Trans\circ D
         = Q
\end{aligned}
\end{equation*}
relies on the following reasoning: First, we multiply by $\id_\vV =
\pi\circ\nu$ and we use $\gamma_{+}\circ\pi\circ\nu =
\pi\circ\gamma^{\,\zeta}\circ\nu$, which follows from \eqref{eq:gamma-zeta}.
Then we multiply by $\id_\vV=D^{-1}\circ\nu^\Trans\circ\pi^\Trans\circ D$.
Applying formula \eqref{eq:D-phi} to $\gamma^{\,\beta}\in
\Oweak(\FV^*,Q^\uparrow)$ and ${D^\uparrow}$ yields $\gamma^{\,\zeta}\circ
{D^\uparrow} = {D^\uparrow}\circ\gamma^{\,\beta} $; compare \eqref{eq:beta}.
Next, we remove $\pi\circ\nu=\id_\vV$. By our assumption,
$Q^\uparrow\circ\gamma^{\,\beta}=Q^\uparrow$. Finally, we cancel
$D^{-1}\circ\nu^\Trans\circ\pi^\Trans\circ D=\id_\vV$.
\end{proof}
\begin{rem}\label{rem:literatur}
Proposition~\ref{prop:Q-auf-ab}~\eqref{prop:Q-auf-ab.a} and
Proposition~\ref{prop:AO-beta=Ow} (with $c:=1$) are sketched without proof in
\cite[p.~102]{helm-05a} using the vector space of all affine functions $\vV\to
F$ rather than $\FV^*$; see Remark~\ref{rem:aff-f}. Analogues of
Proposition~\ref{prop:Q-auf-ab}~\eqref{prop:Q-auf-ab.b} and
Proposition~\ref{prop:AO-beta=Ow} appear in \cite{elle+h-84a} under the extra
assumption $\Char F\neq 2$. Similar results, limited to the case $\dim\vV=2$
and $\Char F\neq 2$, can be found in \cite[\S~8,4--\S~9,4]{bach-73a},
\cite{wolff-67a}, \cite{wolff-67b}, even though the approach from there does
not rely on a weak orthogonal group. See \cite[Sect.~1]{elle-84a} for a
detailed comparison.
\par
The interrelation between the quadratic forms $q_\cP$ and $q_\cL$, as
introduced in \cite[Sect.~2]{stru+s-22a}, does not at all resemble the
interrelation between our $Q$ and $Q^\uparrow$. Nevertheless, for the Euclidean
and Minkowskian planes appearing in \cite{stru+s-22a}, the quadratic form
$q_\cL$ from there admits an interpretation in terms of our $Q^\uparrow$.
\end{rem}

\begin{rem}\label{rem:reflect}
Under the premises of Proposition~\ref{prop:AO-beta=Ow}, the faithful linear
representation $\beta$ maps the set of all affine $Q${\trenn}reflections of
$\bA\VQ$ onto the set of all $Q^\uparrow${\trenn}reflections of
$(\FV^*,Q^\uparrow)$. To be more precise, $\beta$ takes an affine
$Q${\trenn}reflection $\xi_{\vp,\vr}$ as in \eqref{eq:aff-xi} to the
$Q^\uparrow${\trenn}reflection in the direction of $\bigl(-B(\vr,\vp),
D(\vr)\bigr)\in \FV^*$. The straightforward proof, which amounts to
substitutions in \eqref{eq:gamma-beta}, is left to the reader. This bijective
correspondence allows for the translation of theorems about
$Q^\uparrow${\trenn}reflections into theorems about affine
$Q${\trenn}reflections and vice versa.
\par
A correspondence in the same spirit, even though it relies---in our
terminology---on making $\bP(\FV)$ into a \emph{projective metric space}, is
stated in \cite[pp.~164--165]{schroe-92a}.
\end{rem}

\begin{rem}
We maintain the assumptions of Proposition~\ref{prop:AO-beta=Ow}. The quadratic
form $ Q\circ\eps_+^{-1}\colon\{0\}\times\vV \to F \colon (0,\vx)\mapsto
Q(\vx)$, where $\eps$ is given as \eqref{eq:epsilon}, defines a quadric
\cite[p.~964]{schroe-95a} in the hyperplane at infinity of $\bP(\FV)$, namely
\begin{equation*}
    \cF := \bigl\{F(0,\vx)\mid \vx\in\vV\setminus\{\vo\} \mbox{~~and~~} Q(\vx)=0 \bigr\} .
\end{equation*}
This quadric is known as the \emph{absolute quadric} or the \emph{quadric at
infinity} related with the projective embedding of $\bA\VQ$
\cite[p.~267]{burau-61a}, \cite[p.~96]{schroe-92a}. All points of $\cF$ are
simple \cite[p.~965]{schroe-95a}. Likewise, $Q^\uparrow$ defines the quadric
\begin{equation*}
    \cF^\uparrow := \Bigl\{F(a_0,\va^*)\mid (a_0,\va^*)\in\FV^*\setminus\bigl\{(0,\vo^*)\bigr\}
    \mbox{~~and~~} Q^\uparrow(a_0,\va^*)=0 \Bigr\}
\end{equation*}
in $\bP(\FV^*)$. All points of $\cF^\uparrow$ other than $F(1,\vo^*)$ are
simple or, in other words, $\cF^\uparrow$ is a cone with a one-point vertex.
\par
A theorem in \cite[p.~205]{burau-61a} allows us to describe the geometric
relationship between these two quadrics when $F$ is the field of complex
numbers: If $n=\dim\vV\geq 2$, then the points of $\cF^\uparrow$ are the
annihilators of those hyperplanes of\/ $\bP(\FV)$ which contain a tangent space
of $\cF$ with projective dimension $n-2$. This description remains valid in our
more general setting provided that $\cF$ is non-empty. However, the proof of
the underlying theorem, as given \emph{loc.\ cit.}, fails to cover the case
$\Char F=2$.
\par
In the particular case where $\dim\vV=3$ and $\Char F\neq2$, we have a conic
$\cF$ lying in the plane at infinity $\{0\}\times\vV$ of $\bP(\FV)$.
Furthermore, the points of $\cF^\uparrow$ arise as annihilators of the planes
containing a tangent line of $\cF$. See Figure~\ref{abb:dualkegel}, where a few
of these planes are depicted.
\begin{figure}[!h]\unitlength0.9\textwidth
  \centering
  \begin{picture}(1,0.35) 
    \put(0,0){\includegraphics[width=1\unitlength]{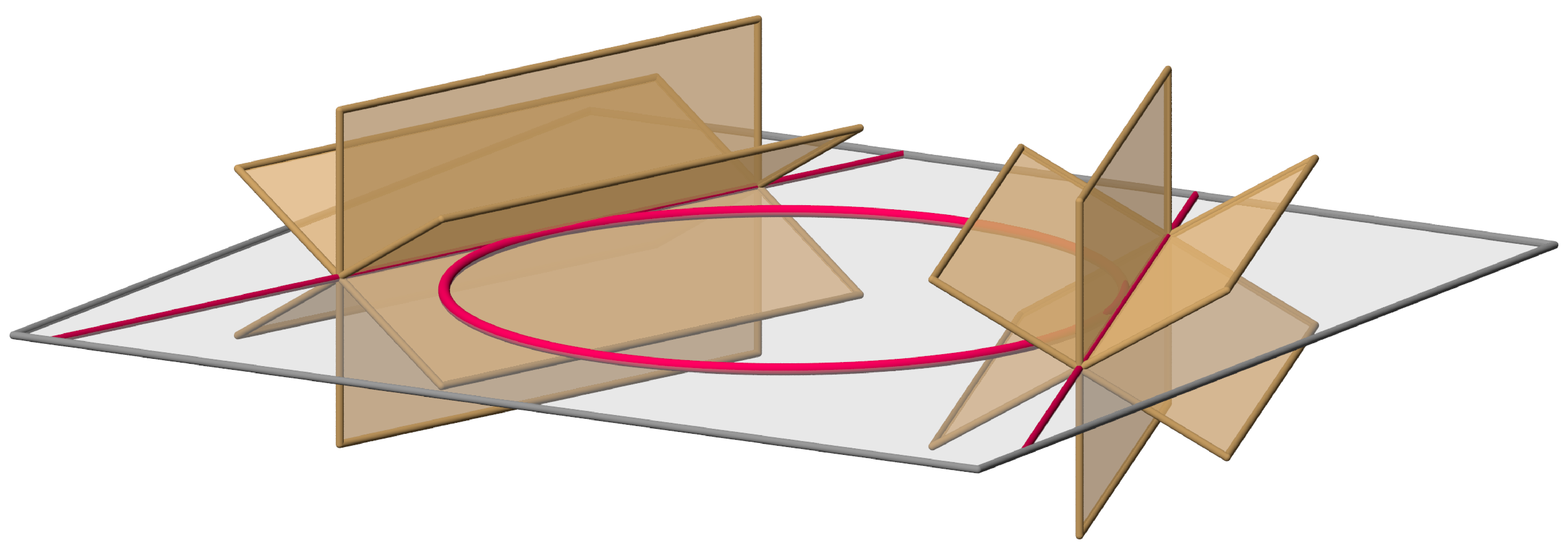}}
    \put(0.5,0.08){$\cF$}
    \put(0.91,0.13){$\{0\}\times\vV$}
  \end{picture}
   \caption{Planes containing a tangent line of $\cF$, $\dim\vV=3$ and $F=\bR$}
   \label{abb:dualkegel}
\end{figure}
\end{rem}

\begin{rem}\label{rem:modif}
Let $\Char F\neq 2$. Then, due to $\frac{1}{2} B(\vx,\vx)=Q(\vx)$ for all
$\vx\in\vV$, it is highly common to associate with $Q$ the bilinear form
$\frac{1}{2} B$ rather than $B$. By doing so, the mapping $\frac{1}{2} D$ takes
over the role of $D$. This suggests a variant form of
Proposition~\ref{prop:Q-auf-ab} by considering the pullback of $Q$ along
$\bigl(\frac{1}{2}D\bigr)^{-1}\circ\nu^\Trans$. From $\bigl(\frac{1}{2}
D\bigr)^{-1} = 2 D^{-1}$, in this way the quadratic form $4Q^\uparrow$ is being
linked with $Q$. Proposition~\ref{prop:AO-beta=Ow} covers this variant by
putting $c:=4$.
\end{rem}

\begin{rem}\label{rem:koo}
Let us write down the transition from $Q$ to $Q^\uparrow$, as in
Proposition~\ref{prop:Q-auf-ab}, in terms of coordinates. Upon choosing any
basis of $\vV$, say
\begin{equation}\label{eq:V-basis}
    \{\ve_1,\ve_2,\ldots,\ve_n\} ,
\end{equation}
we denote the corresponding dual basis of $\vV^*$ by
\begin{equation}\label{eq:V-dualbasis}
    \{\ve_1^*,\ve_2^*,\ldots,\ve_n^*\} .
\end{equation}
Then $\FV$ admits the basis
$\bigl\{(1,\vo),(0,\ve_1),(0,\ve_2),\ldots,(0,\ve_n)\bigr\}$, which has
\begin{equation}\label{eq:FV-dualbasis}
    \bigl\{(1,\vo^*),(0,\ve_1^*),(0,\ve_2^*),\ldots,(0,\ve_n^*)\bigr\}
\end{equation}
as its dual basis. There is at least one matrix $W=(w_{ij})$ with entries in
$F$ and $i,j$ ranging in $\{1,2,\ldots,n\}$ such that
\begin{equation}\label{eq:Q-koo}
    Q\left(\sum_{h=1}^{n} x_{h}\ve_{h}\right)
    = \sum_{i=1}^{n}\sum_{j=1}^{n} w_{ij} x_{i}x_{j}
    \mbox{~~for all~~}
    x_{1},x_{2},\ldots,x_{n}\in F .
\end{equation}
The matrix of $B$ relative to the basis \eqref{eq:V-basis} reads $W+W^\Trans$,
where $W^\Trans$ denotes the transpose of $W$. From $B$ being
non{\trenn}degenerate, $W+W^\Trans$ turns out invertible. We define the block
diagonal matrix\footnote{Note that the second block is congruent to $W$, since
$(W+W^\Trans)^{-1}$ is a symmetric matrix.}
\begin{equation*}
    W^\uparrow:=\diag\left(0,(W+W^\Trans)^{-1} \cdot W \cdot (W+W^\Trans)^{-1}\right)
\end{equation*}
with row and column indices of $W^\uparrow$ ranging in $\{0,1,\ldots,n\}$. The
$(i,j)$-entry of $W^\uparrow$ will be written as $w^\uparrow_{ij}$. Since
$(W+W^\Trans)^{-1}$ describes $D^{-1}$ relative to the bases
\eqref{eq:V-dualbasis} and \eqref{eq:V-basis}, equations \eqref{eq:Q-auf} and
\eqref{eq:Q-koo} yield
\begin{equation*}
    Q^\uparrow\left(a_0(1,\vo^*)+\sum_{h=1}^{n} a_h(0,\ve_h^*)\right)
    = \sum_{i=0}^{n}\sum_{j=0}^{n} w^\uparrow_{ij} a_{i}a_{j}
    \mbox{~~for all~~}
    a_{0},a_{1},\ldots,a_{n}\in F .
\end{equation*}
\par
If $\Char F\neq 2$, then the above calculation can be simplified by choosing
$W$ as a \emph{symmetric} matrix. So, $W+W^\Trans=2W$ and one readily
verifies\footnote{As we observed in Remark~\ref{rem:modif}, using $4Q^\uparrow$
simplifies matters when $\Char F\neq 2$.} $4W^\uparrow = \diag(0, W^{-1})$.
Moreover, we may start in \eqref{eq:V-basis} with an \emph{orthogonal} basis of
$\VQ$, which makes the symmetric matrix $W$ a diagonal matrix and simplifies
the calculation of $W^{-1}$. In particular, over the real numbers there is a
choice of \eqref{eq:V-basis} with
\begin{equation*}
    W = \diag(\underbrace{1,1,\ldots,1}_{p},\underbrace{-1,-1,\ldots,-1}_{n-p}) = W^{-1} ,
\end{equation*}
that is, $\frac{1}{2}B$ (and likewise $B$) has signature $(p,n-p,0)$. Then
\begin{equation*}
    4 W^\uparrow = \diag(0,W) = \diag(0,\underbrace{1,1,\ldots,1}_{p},\underbrace{-1,-1,\ldots,-1}_{n-p}).
\end{equation*}
This illustrates how the work about the real case fits into our approach; see
the references at the beginning of Section~\ref{se:intro}.
\par
If $\Char F = 2$, then $B$ is a non{\trenn}degenerate alternating bilinear form
and so $\dim\vV=n$ has to be even. We therefore are in a position to choose the
basis \eqref{eq:V-basis} in such a way that the (alternating) $n\times n$
matrix of $B$ relative to \eqref{eq:V-basis} takes the
\emph{block{\trenn}diagonal form}
\begin{equation*}
\diag\left(\left(\begin{array}{cc}
                     0 & 1 \\
                     1 & 0 \\
                 \end{array}\right),
                 \left(\begin{array}{cc}
                     0 & 1 \\
                     1 & 0 \\
                 \end{array}\right),
                 \ldots,
                 \left(\begin{array}{cc}
                     0 & 1 \\
                     1 & 0 \\
                 \end{array}\right)
          \right) .
\end{equation*}
Next, we may select an \emph{upper triangular matrix} $W$ subject to
\eqref{eq:Q-koo}. As $W+W^\Trans$ equals the above-noted block diagonal matrix,
we have
\begin{equation*}
    W=\diag\left(\left(\begin{array}{cc}
                     w_{11}& 1 \\
                     0 & w_{22} \\
                 \end{array}\right),
                 \left(\begin{array}{cc}
                     w_{33} & 1 \\
                     0 & w_{44} \\
                 \end{array}\right),
                 \ldots,
                 \left(\begin{array}{cc}
                     w_{n-1,n-1} & 1 \\
                     0 & w_{nn} \\
                 \end{array}\right)
          \right) .
\end{equation*}
From $W+W^\Trans$ being self{\trenn}inverse, we end up with
\begin{equation*}
    W^\uparrow =\diag\left(0,
                \left(\begin{array}{cc}
                     w_{22}& 0 \\
                     1 & w_{11} \\
                 \end{array}\right),
                 \left(\begin{array}{cc}
                     w_{44} & 0 \\
                     1 & w_{33} \\
                 \end{array}\right),
                 \ldots,
                 \left(\begin{array}{cc}
                     w_{nn} & 0 \\
                     1 & w_{n-1,n-1} \\
                 \end{array}\right)
          \right) .
\end{equation*}
\end{rem}

We now turn to the problem of describing \emph{all} solutions of the problem
posed at the beginning of this section. The following corollary to
Lemma~\ref{lem:trans} will be a powerful tool, since it does not involve a
quadratic form $Q\colon\vV\to F$.

\begin{cor}\label{cor:trans}
Let $\vV$ be a vector space and let $\tilde{Q}\colon\FV^*\to F$ be a quadratic
form. Then, with $\beta$ as in \eqref{eq:beta}, the following are equivalent.
\begin{enumerate}\itemsep0pt\parsep0pt
\item\label{cor:trans.a} The $\beta$-image of the translation group of
    $\vV$ is contained in $\Oweak(\FV^*,\tilde{Q})$.
\item\label{cor:trans.b} One of the subsequent conditions holds:
    \begin{gather}
        \tilde{Q}(1,\vo^*)=0 \mbox{~~and~~} (\FV^*)^{\tilde{\perp}} = F(1,\vo^*);
            \label{eq:cor-trans}\\
        \dim(\FV^*) = 1 ;
            \label{eq:cor-trans=1}\\
        \dim(\FV^*) = 2 ,\; \tilde{Q}(1,\vo^*)\neq 0 ,\;
        \dim(\FV^*)^{\tilde{\perp}} = 0 \mbox{~~and~~} |F| = 2.
            \label{eq:cor-trans=2}
    \end{gather}
\end{enumerate}
\end{cor}
\begin{proof}
Let any translation $\gamma\in\AGL(\vV)$ be given as in \eqref{eq:gamma}, that
is with $\gamma_+ = \id_\vV$ and $\vt\in\vV$. A comparison of
\eqref{eq:gamma-beta} with \eqref{eq:delta} readily shows
\begin{equation*}
    \gamma^{\,\beta} =
    \delta_{(0,-\vt),(1,\vo^*)}\in \Delta\bigl(\FV^*,(1,\vo^*)\bigr) ;
\end{equation*}
see also \eqref{eq:Delta}. The annihilator $\bigl\{(1,\vo^*)\bigr\}^\circ \leq
\FV$ comprises precisely the vectors $(x_0,\vx)$ with $x_0 = 0$ and
$\vx\in\vV$. So, the $\beta$-image of the translation group of $\vV$ equals
\begin{equation*}
    \Bigl\{\delta_{(x_0,\vx),(1,\vo^*)}\mid (x_0,\vx)\in
\bigl\{(1,\vo^*)\bigr\}^\circ\Bigr\} .
\end{equation*}
By the above, the proof is reduced to a rewording of Lemma~\ref{lem:trans}: In
the present context the ``non-zero vector $\vf$'' and the ``metric vector space
$\VQ$'' from there have to be replaced with the ``linear form $(1,\vo^*)$'' and
the ``metric vector space $(\FV^*,\tilde{Q})$'', respectively.
\end{proof}

Next, we present our first main result.

\begin{thm}\label{thm:umkehr}
Let $\VQ$ and $(\FV^*,\tilde{Q})$ be metric vector spaces such that, with
$\beta$ as in \eqref{eq:beta}, one of the equations \eqref{eq:oder1} or
\eqref{eq:oder2} is satisfied. Furthermore, suppose that none of the subsequent
conditions applies:
\begin{align}
    \dim\vV=0 \mbox{~~and~~}& {\Char F = 2} ; \label{eq:dim=0} \\
    \dim\vV=1 \mbox{~~and~~}& |F| \leq 3 ; \label{eq:dim=1} \\
    \dim\vV=2 \mbox{~~and~~}& |F| = 2 .    \label{eq:dim=2}
\end{align}
Then the following hold:
\begin{enumerate}\itemsep0pt\parsep0pt
\item\label{thm:umkehr.a} The polar form of $Q$ is non{\trenn}degenerate.

\item\label{thm:umkehr.b} The quadratic form $\tilde{Q}$ is a non-zero
    scalar multiple of that quadratic form $Q^\uparrow$ which arises from
    $Q$ according to \eqref{eq:Q-auf}.
\end{enumerate}
\end{thm}
\begin{proof}
\emph{Case} 1: $\dim\vV = 0$. Then \eqref{thm:umkehr.a} holds trivially, since
$Q$ is a zero quadratic form. Furthermore, $Q^\uparrow\colon \FV^*\to F$ is
also a zero quadratic form. Since \eqref{eq:dim=0} does not apply, the weak
orthogonal group of any non-zero quadratic form on $\FV^*$ contains
$-\id_{\FV^*}\notin \AOrth\VQ^{\,\beta} = \{\id_{\FV^*}\}$. Thus $\tilde{Q} =
Q^\uparrow$, which verifies \eqref{thm:umkehr.b}.
\par
\emph{Case} 2: $\dim\vV\geq 1$. By our assumptions, the $\beta$-image of the
translation group of $\vV$ is contained in $\Oweak(\FV^*,\tilde{Q})$. So, from
Corollary~\ref{cor:trans} and from \eqref{eq:dim=1} being ruled out, it follows
that \eqref{eq:cor-trans} is satisfied. Consequently, $\tilde{Q}^\downarrow
\colon \vV \to F$ is well defined and its polar form $\tilde{B}^\downarrow$ is
non{\trenn}degenerate; see
Proposition~\ref{prop:Q-auf-ab}~\eqref{prop:Q-auf-ab.b}. Next, we employ
Proposition~\ref{prop:AO-beta=Ow} on $(\vV, \tilde{Q}^\downarrow)$. Since
$(\tilde{Q}^\downarrow)^\uparrow = \tilde{Q}$ according to
Proposition~\ref{prop:Q-auf-ab}~\eqref{prop:Q-auf-ab.d}, we have
$\AOrth(\vV,\tilde{Q}^\downarrow)^{\,\beta} = \Oweak(\FV^*,\tilde{Q})$. Going
over to $\beta${\trenn}preimages gives, by \eqref{eq:oder1} or
\eqref{eq:oder2}, $\AOrth(\vV,\tilde{Q}^\downarrow) = \AOrth\VQ$ or
$\AOrth(\vV,\tilde{Q}^\downarrow) = \AOweak\VQ$. Consequently,
\begin{equation*}
    \Orth(\vV,\tilde{Q}^\downarrow) = \Orth\VQ \mbox{~~or~~}
    \Orth(\vV,\tilde{Q}^\downarrow) = \Oweak\VQ  .
\end{equation*}
From the above statement and due to the exclusion of the cases appearing in
\eqref{eq:dim=1} and \eqref{eq:dim=2}, we are now in a position to make use of
Lemma~\ref{lem:Orth1=2} with $\tilde{Q}^\downarrow$ and $Q$ playing the roles
of the quadratic forms $Q_1$ and $Q_2$, respectively. So, there is a $c\in
F^\times$ with $cQ = \tilde{Q}^\downarrow$, which establishes
\eqref{thm:umkehr.a}. The last equation together with
Proposition~\ref{prop:Q-auf-ab} gives $cQ^\uparrow = (cQ)^\uparrow =
(\tilde{Q}^\downarrow)^\uparrow = \tilde{Q}$, that is, \eqref{thm:umkehr.b} is
satisfied too.
\end{proof}

Theorem~\ref{thm:umkehr} shows---loosely speaking---that in general $\tilde{Q}$
is determined by $Q$ up to a non-zero scalar factor.

\begin{rem}\label{rem:tab}
In Tables~\ref{tab:1}--\ref{tab:4}, we summarise all dyads of metric vector
spaces $\VQ$ and $(\FV^*,\tilde{Q})$ satisfying \eqref{eq:oder1} or
\eqref{eq:oder2} \emph{without} being covered by Theorem~\ref{thm:umkehr}. We
thereby make use of a fixed basis of $\vV$ as in \eqref{eq:V-basis} together
with its dual basis \eqref{eq:V-dualbasis}, the basis \eqref{eq:FV-dualbasis}
of $\FV^*$ and coordinates $x_1,x_2,\ldots,x_n,a_0,a_1,\ldots,a_n\in F$.
\par
Each table is to be read in two ways: First, quadratic forms $Q$ and
$\tilde{Q}$ appear in the \emph{same block}, that is to mean between two
adjacent horizontal lines, precisely when \eqref{eq:oder1} or \eqref{eq:oder2}
applies. Take notice that in all these instances equation \eqref{eq:oder1} is
satisfied. So, all quadratic forms $Q$ from within the same block share a
common orthogonal group and all quadratic forms $\tilde{Q}$ from within the
same block share a common weak orthogonal group. Second, quadratic forms $Q$
and $\tilde{Q}$ appear in the \emph{same row} if, and only if, $\tilde{Q} =
Q^\uparrow$ or, equivalently, $Q = \tilde{Q}^\downarrow$. A quadratic form
appears in a \emph{row with one entry left in blank} if, and only if, it fails
to meet the corresponding hypotheses of
Proposition~\ref{prop:Q-auf-ab}~\eqref{prop:Q-auf-ab.a} or
\eqref{prop:Q-auf-ab.b}.
\par
There are but a few cases, where $\AOweak\VQ$ is a \emph{proper subgroup} of
$\AOrth\VQ$. This happens precisely when $Q$ is given as in the last row of
Table~\ref{tab:3} or as in the second, fourth, sixth or eighth row of
Table~\ref{tab:4}.
\par
Let $\dim\vV=0$ and $\Char F = 2$. The first row of Table~\ref{tab:1} arises
from that solution of \eqref{eq:oder1}, where both $Q$ and
$\tilde{Q}=Q^\uparrow$ are zero quadratic forms. Due to $\Char F=2$, any
non-zero quadratic form on $\FV^*$ has $\{\id_{\FV^*}\}$ as its weak orthogonal
group. Thus, together with the zero quadratic form on $\vV$, it provides a
solution of \eqref{eq:oder1}. Therefore, \emph{all} non-zero quadratic forms on
$\FV^*$ are listed on the right hand side of the table's second row.
\begin{table}[h!]\centering\renewcommand{\arraystretch}{1.4}
    \begin{tabular}{|p{0.3\textwidth}|p{0.3\textwidth}|}
        \hline
        $Q(\vo)$ & $\tilde{Q}(a_0,\vo^*)$              \\
        \hline\hline
        $=0$     & $=0$                                \\
                 & $=\tilde{w}_{00}a_0^2$ (with $\tilde{w}_{00}\in F^\times$)    \\
        \hline
    \end{tabular}
    \caption{$\dim\vV=0$ and $\Char F=2$}\label{tab:1}
\end{table}
\par
Let $\dim\vV=1$ and $|F| = 2$. If $Q$ and $\tilde{Q}$ satisfy \eqref{eq:oder1}
or \eqref{eq:oder2}, then $\tilde{Q}$ meets the requirements of
Corollary~\ref{cor:trans}~\eqref{cor:trans.a}. Consequently, at least one of
the three conditions of Corollary~\ref{cor:trans}~\eqref{cor:trans.b} applies.
Take notice that \eqref{eq:cor-trans} cannot be met, since $\dim(\FV^*)=2$ and
$\Char F = 2$ forces $\dim (\FV^*)^{\tilde{\perp}}$ to be even. Equation
\eqref{eq:cor-trans=1} fails obviously. Thus we are led to
\eqref{eq:cor-trans=2}, which is satisfied by precisely two quadratic forms.
The first one appears in Table~\ref{tab:2} on the right hand side of the first
row, since it provides a solution of \eqref{eq:oder1} together with the
quadratic forms on $\vV$, as listed on the left hand side of the second and
third row. The second one is given by $(a_0,a_1\ve_1^*)\mapsto
a_0^2+a_0a_1+a_1^2$ for all $a_0,a_1\in F$. But its weak orthogonal group,
which equals $\GL(\FV^*)$, is not contained in $\AGL(\vV)^{\,\beta}$ due to
$6=\bigl|{\GL(\FV^*)}\bigr| > \bigl|{\AGL(\vV)^{\,\beta}}\bigr| = 2$. So, this
second quadratic form does not appear in Table~\ref{tab:2}.
\begin{table}[h!]\centering\renewcommand{\arraystretch}{1.4}
    \begin{tabular}{|p{0.3\textwidth}|p{0.3\textwidth}|}
        \hline
        $Q(x_1\ve_1)$ & $\tilde{Q}(a_0,a_1\ve_1^*)$ \\
        \hline\hline
                      & $=a_0^2+a_0a_1$ \\
        $=0$          &                 \\
        $=x_1^2$      &                 \\
        \hline
    \end{tabular}
    \caption{$\dim\vV=1$ and $|F|=2$}\label{tab:2}
\end{table}
\par
Let $\dim\vV=1$ and $|F|=3$. The first two rows of Table~\ref{tab:3} contain
all solutions $Q,\tilde{Q}$ of \eqref{eq:oder1} with $\tilde{Q} = Q^\uparrow$.
The zero quadratic form on $\vV$ appears on the left hand side of the third
row, since it satisfies \eqref{eq:oder1} together with each of the two
quadratic forms on the right hand side. So, the left column of
Table~\ref{tab:3} comprises \emph{all} quadratic forms on $\vV$. As we noted
above, a quadratic form $\tilde{Q}$ may only appear in the right column if it
meets one of the three conditions of
Corollary~\ref{cor:trans}~\eqref{cor:trans.b}. Since neither
\eqref{eq:cor-trans=1} nor \eqref{eq:cor-trans=2} can be satisfied, there
remains \eqref{eq:cor-trans}. All $\tilde{Q}$ subject to \eqref{eq:cor-trans}
are listed in the right column of the table, as follows from
Proposition~\ref{prop:Q-auf-ab}. Consequently, the right column already has
been filled up completely.
\begin{table}[h!]\centering\renewcommand{\arraystretch}{1.4}
    \begin{tabular}{|p{0.3\textwidth}|p{0.3\textwidth}|}
        \hline
        $Q(x_1\ve_1)$ & $\tilde{Q}(a_0, a_1 \ve_1^*)$  \\
        \hline\hline
        $=x_1^2$      &  $=a_1^2$   \\
        $=-x_1^2$     &  $=-a_1^2$  \\
        $=0$          &             \\
        \hline
    \end{tabular}
    \caption{$\dim\vV=1$ and $|F|=3$}\label{tab:3}
\end{table}
\par
Let $\dim\vV=2$ and $|F|=2$. There are four solutions $Q,\tilde{Q}$ of
\eqref{eq:oder1} with $\tilde{Q} = Q^\uparrow$. They can be read off from the
first, third, fifth and seventh row of Table~\ref{tab:4}. We put these rows
into four different blocks, since the corresponding quadratic forms on $\vV$
have mutually distinct orthogonal groups; see below. There are no more entries
in the right column: This follows, as in the previous cases, since no quadratic
form on $\FV^*$ satisfies one of the conditions \eqref{eq:cor-trans=1} or
\eqref{eq:cor-trans=2} of Corollary~\ref{cor:trans}~\eqref{cor:trans.b}.
However, it turns out that in each block there is precisely one more entry on
the left hand side. The two quadratic forms on $\vV$ from the first block have
$\GL(\vV)$ as their common orthogonal group. The two quadratic forms on $\vV$
from any of the remaining blocks also share a common orthogonal group, namely
the stabiliser in $\GL(\vV)$ of a particular non-zero vector. More precisely,
for the second, third and fourth block, this vector reads $\ve_1+\ve_2$,
$\ve_1$ and $\ve_2$, respectively.
\begin{table}[h!]\centering\renewcommand{\arraystretch}{1.4}
    \begin{tabular}{|p{0.3\textwidth}|p{0.3\textwidth}|}
        \hline
        $Q(x_1\ve_1+x_2\ve_2)$ & $\tilde{Q}(a_0,a_1\ve_1^*+a_2 \ve_2^*)$ \\
        \hline\hline
        $=x_1^2+x_1x_2+x_2^2$ & $=a_1^2+a_1a_2+a_2^2$            \\
        $=0$                  &                                  \\
        \hline
        $=x_1x_2$             & $=a_1a_2$                        \\
        $=x_1^2+x_2^2$        &                                  \\
        \hline
        $=x_1^2+x_1x_2$       & $=a_1^2+a_1a_2$                  \\
        $=x_2^2$              &                                  \\
        \hline
        $=x_1x_2+x_2^2$       & $=a_1a_2+a_2^2$                  \\
        $=x_1^2$              &                                  \\
        \hline
    \end{tabular}
    \caption{$\dim\vV=2$ and $|F|=2$}\label{tab:4}
\end{table}
\end{rem}

To close this paper, let us address another question: Is there a way to
describe, with the techniques at our disposal, the motion group or the weak
motion group of additional affine metric spaces $\bA\VQ$ by changing over to
the projective space $\bP(\FV^*)$? The idea behind is that in this way one
gains an additional ``degree of freedom'', since a projective collineation of
$\bP(\FV^*)$ onto itself is induced by $|F^\times|$ linear bijections differing
by non-zero scalar factors. Furthermore, nothing is lost by the transition to
$\bP(\FV^*)$, since the linear representation $\beta$ of $\AGL(\vV)$, as in
\eqref{eq:beta}, has a property that goes beyond its being faithful: The
$\beta$-images of distinct affinities are not proportional and so they act
differently on $\bP(\FV^*)$.
\par
We proceed by writing up a version of Lemma~\ref{lem:trans-skal} in the same
way as Corollary~\ref{cor:trans} resembles Lemma~\ref{lem:trans}. Then we
answer the raised question.

\begin{cor}\label{cor:trans-skal}
Let $\vV$ be a vector space and let $\tilde{Q}\colon\FV^*\to F$ be a quadratic
form. Then, with $\beta$ as in \eqref{eq:beta}, for all $s\in
F\setminus\{0,1\}$ and all non{\trenn}identical translations
$\gamma\in\AGL(\vV)$, the mapping $s\gamma^{\,\beta}$ does not belong to the
weak orthogonal group $\Oweak(\FV^*,\tilde{Q})$.
\end{cor}

\begin{thm}\label{thm:proj}
Let $\VQ$ and $(\FV^*, \tilde{Q})$ be metric vector spaces such that
$\Oweak(\FV^*,\tilde{Q})$ induces the same collineation group on the projective
space $\bP(\FV^*)$ as one of the groups $\AOrth\VQ^{\,\beta}$ or
$\AOweak\VQ^{\,\beta}$, where $\beta$ is given as in \eqref{eq:beta}. Then
$\AOrth\VQ^{\,\beta} = \Oweak(\FV^*,\tilde{Q}) $, unless $\dim\vV=0$, $\Char
F\neq 2$ and $\tilde{Q}(\FV^*)\neq\{0\}$.
\end{thm}
\begin{proof}
\emph{Case} 1: $\dim\vV=0$. Then $\AOrth\VQ^{\,\beta} = \AOweak\VQ^{\,\beta} =
\{\id_\vV\}^{\,\beta} = \{\id_{\FV^*}\}$. If $\Char F=2$ or
$\tilde{Q}(\FV^*)=\{0\}$, then the radical of $\tilde{B}$ equals $\FV^*$ and so
$\Oweak(\FV^*,\tilde{Q}) = \{\id_{\FV^*}\}$. If $\Char F\neq 2$ and
$\tilde{Q}(\FV^*)\neq\{0\}$, then the groups $\Oweak(\FV^*,\tilde{Q}) =
\{\pm\id_{\FV^*}\}$ and $\{\id_{\FV^*}\}$ are distinct, even though they
determine the same (trivial) group of collineations on $\bP(\FV^*)$.
\par
\emph{Case} 2: $\dim\vV\geq 1$. By our assumptions, for each
$\gamma\in\AOrth\VQ$ or for each $\gamma\in\AOweak\VQ$ there is at least one
scalar $s_\gamma \in F^\times$ such that
$s_\gamma\gamma^{\,\beta}\in\Oweak(\FV^*,\tilde{Q})$. We claim that any such
$s_\gamma$ has to be $1\in F$. This is obvious when $|F|=2$. Up to the end of
the current paragraph, we therefore assume $|F|\geq 3$. If $\gamma$ is an
arbitrary non{\trenn}trivial translation of $\vV$, then
$\gamma^{\,\beta}\in\Oweak(\FV^*)$ follows from Corollary~\ref{cor:trans-skal}.
This implies, together with $\id_\vV^{\,\beta}\in\Oweak(\FV^*,\tilde{Q})$, that
condition \eqref{cor:trans.a} of Corollary~\ref{cor:trans} is satisfied. Hence
the equivalent condition \eqref{cor:trans.b} from there is satisfied too. Since
\eqref{eq:cor-trans=1} and \eqref{eq:cor-trans=2} are false in the present
setting, we read off from \eqref{eq:cor-trans} that the radical of $\tilde{B}$
equals $F(1,\vo^*)$. Now, returning to an arbitrary $\gamma$ as described
above, Lemma~\ref{lem:AGL-rep} gives
$s_\gamma\gamma^{\,\beta}(1,\vo^*)=(s_\gamma,\vo^*)$, whereas
$s_\gamma\gamma^{\,\beta}\in\Oweak(\FV^*,\tilde{Q})$ forces
$s_\gamma\gamma^{\,\beta}(1,\vo^*)=(1,\vo^*)$. Therefore $s_\gamma=1$.
\par
By the above, the hypotheses of Theorem~\ref{thm:umkehr} are fulfilled provided
that neither \eqref{eq:dim=1} nor \eqref{eq:dim=2} is satisfied. So, up to
these cases, Theorem~\ref{thm:umkehr}~\eqref{thm:umkehr.a} gives the even
stronger result $\AOrth\VQ^{\,\beta} = \AOweak\VQ^{\,\beta}
=\Oweak(\FV^*,\tilde{Q})$. Otherwise, the claim follows from
Tables~\ref{tab:2}, \ref{tab:3} and \ref{tab:4} in Remark~\ref{rem:tab}.
\end{proof}

All things considered, up to a single trivial case, adopting the projective
point of view fails to significantly amplify the scope of our approach.
\begin{rem}
Tables~1 and 3 in \cite{klaw+h-13a} (and likewise Tables~3.1 and 3.3 in
\cite[pp.~192--195]{klaw-15a}) about metric vector spaces over $\bR$ appear to
be partially incorrect. The reason is that these tables contain---using our
terminology---several dyads of metric vector spaces $\VQ$ and
$(\FV^*,\tilde{Q})$, where the polar form of $Q$ is degenerate and $F=\bR$.
Moreover, it is claimed (without giving a formal proof) that these dyads
satisfy the assumptions of Theorem~\ref{thm:proj}, which seems impossible by
the above proof and Theorem~\ref{thm:umkehr}.
\end{rem}

\small

\noindent
Hans Havlicek\\
Institut f\"{u}r Diskrete Mathematik und Geometrie\\
Technische Universit\"{a}t Wien\\
Wiedner Hauptstra{\ss}e 8--10/104\\
1040 Wien\\
Austria\\
\texttt{havlicek@geometrie.tuwien.ac.at}
\end{document}